\documentclass[english,10pt]{amsart}

\usepackage[english]{babel}

\usepackage{amscd,amssymb,amsfonts,amsmath}
\usepackage{tikz-cd}

\usepackage{bm}
\usepackage{graphics}
\usepackage{epsfig}
\usepackage{mathrsfs}
\usepackage{xcolor}
\DeclareMathAlphabet{\mathscrbf}{OMS}{mdugm}{b}{n}

\definecolor{violet}{rgb}{0.0,0.2,0.7}
\definecolor{rouge2}{rgb}{0.8,0.0,0.2}
\usepackage{hyperref}
\hypersetup{
    unicode=false,          
    pdftoolbar=true,        
    pdfmenubar=true,        
    pdffitwindow=false,     
    pdfstartview={FitH},    
    pdftitle={},    
    pdfauthor={},     
    colorlinks=true,       
   linkcolor=violet,          
    citecolor=rouge2,        
    filecolor=black,      
    urlcolor=cyan}           

\usepackage[text={6.7in,9.2in},centering]{geometry}

\makeatletter
\renewcommand\subsection{\@startsection{subsection}{2}%
  \z@{.5\linespacing\@plus.7\linespacing}{-.5em}%
  {\normalfont\sffamily}}  
\makeatother

\newcommand{\codim}{\textup{codim}}

\renewcommand{\phi}{\varphi}
\newcommand{\into}{\hookrightarrow}
\newcommand{\map}{\dashrightarrow}

\newcommand{\wh}{\widehat}
\newcommand{\wb}{\overline}

\renewcommand{\le}{\leqslant}
\renewcommand{\ge}{\geqslant}

\newcommand{\sA}{\mathscr{A}}

\newcommand{\sE}{\mathscr{E}}
\newcommand{\sF}{\mathscr{F}}
\newcommand{\sG}{\mathscr{G}}

\newcommand{\sI}{\mathscr{I}}

\newcommand{\sL}{\mathscr{L}}
\newcommand{\sM}{\mathscr{M}}
\newcommand{\sN}{\mathscr{N}}
\newcommand{\sO}{\mathscr{O}}
\newcommand{\sP}{\mathscr{P}}

\newcommand{\Der}{\textup{Der}}

\newtheorem{thm}{Theorem}[section]

\newtheorem{lemma}[thm]{Lemma}
\newtheorem{cor}[thm]{Corollary}

\newtheorem{prop}[thm]{Proposition}

\newtheorem*{thm*}{Theorem}
\theoremstyle{definition}
\newtheorem{defn}[thm]{Definition}

\newtheorem{defn-thm}[thm]{Definition-Theorem} 
\newtheorem{defn-lemma}[thm]{Definition-Lemma}

\theoremstyle{remark}
\newtheorem{claim}[thm]{Claim}
\newtheorem{fact}[thm]{Fact}

\newtheorem*{not-and-def}{Notation and definitions}
 
\newtheorem{rem}[thm]{Remark}

\newtheorem{exmp}[thm]{Example}
\newtheorem{setup}[thm]{Setup}

\numberwithin{equation}{section}

\def\factor#1.#2.{\left. \raise 2pt\hbox{$#1$} \right/\hskip -2pt\raise -2pt\hbox{$#2$}}

\begin{document} 

\title[Projectively flat log smooth pairs]{Projectively flat log smooth pairs}

\author{St\'ephane \textsc{Druel}}

\address{St\'ephane Druel: Univ Lyon, CNRS, Universit\'e Claude Bernard Lyon 1, UMR 5208, Institut Camille Jordan, F-69622 Villeurbanne, France} 

\email{stephane.druel@math.cnrs.fr}


\begin{abstract}
In this article, we study projective log smooth pairs with numerically flat normalized logarithmic tangent bundle. Generalizing works of Jahnke-Radloff and Greb-Kebekus-Peternell, we show that, passing to an appropriate finite cover and up to isomorphism, these are the projective spaces or the log smooth pairs with numerically flat logarithmic tangent bundles blown-up at finitely many points away from the boundary. On the other hand, the structure of log smooth pairs with numerically flat logarithmic tangent bundle is well understood: they are toric fiber bundles over finite \'etale quotients of abelian varieties.
\end{abstract}

\maketitle

{\small\tableofcontents}

\section{Introduction}
Let $X$ be a smooth projective algebraic variety over the field of complex numbers and let $D\subset X$ be a divisor with normal crossings.

The structure of pairs $(X,D)$ with trivial logarithmic tangent bundle $T_X(-\textup{log}\, D)$ is well understood by a result of Winkelmann (see \cite{winkelmann_log_trivial}). They are called semiabelic varieties.
The simplest examples are pairs $(A,0)$ where $A$ is an abelian variety, and pairs $(X,D)$ where $X$ is a smooth toric variety with boundary divisor $D$. If $(X,D)$ is a semiabelic variety, then the algebraic group $G:=\textup{Aut}^{0}(X,D)$ is a semiabelian group which acts on $X$ with finitely many orbits. Moreover, the $G$-orbits in $X$ are exactly the strata defined by $D$. As a consequence, the albanese map is a smooth locally trivial fibration with typical fiber $F$ being a toric variety with boundary divisor $D|_F$. 

If $T_X(-\textup{log}\, D)$ is only assumed to be numerically flat, then it has been shown in \cite[Corollary 1.7]{druel_lo_bianco} that there is a smooth morphism $a\colon X \to A$ with connected fibers onto a finite \'etale quotient of an abelian variety. Moreover, the fibration $(X,D)\to A$ is locally trivial for the analytic topology and any fiber $F$ of the map $a$ is a smooth toric variety with boundary divisor $D|_F$.

In this article, we address pairs $(X,D)$ whose normalized logarithmic tangent bundle $\textup{S}^{n} T_X(-\textup{log}\,D)\otimes\sO_X(-(K_X+D))$ is numerically flat, where $n:=\dim X$. Recall from \cite[Theorem 1.1]{jahnke_radloff_13} that the vector bundle $\textup{S}^{n} T_X(-\textup{log}\,D)\otimes\sO_X(-(K_X+D))$ is numerically flat if and only if $T_X(-\textup{log}\,D)$ is semistable with respect to some ample divisor $H$ and equality holds in the Bogomolov-Gieseker inequality,
$$\frac{n-1}{2n}c_1(T_X(-\textup{log}\,D))^2\cdot H^{n-2}=c_2(T_X(-\textup{log}\,D))\cdot H^{n-2}.$$
The simplest examples of pairs $(X,D)$ with numerically flat normalized logarithmic tangent bundle are those with  
numerically flat logarithmic tangent bundle and the pairs $(\mathbb{P}^n,H)$ where $H$ is an hyperplane. 
Moreover, a smooth finite cover which is \'etale over $X \setminus D$ or the blow-up at a point in $X \setminus D$ of a pair 
$(X,D)$ with numerically flat normalized logarithmic tangent bundle has numerically flat normalized logarithmic tangent bundle as well.

\begin{thm}\label{thm_intro:main}
Let $(X,D)$ be a reduced log smooth pair with $X$ a complex projective variety of dimension $n \ge 2$. Suppose that the normalized vector bundle $\textup{S}^n T_X(-\textup{log}\,D)\otimes\sO_X(-(K_X+D))$ is numerically flat. 
Then there exist a smooth projective variety $Y$ and a log smooth reduced pair $(Z,B)$ as a well as a finite cover $\gamma\colon Y \to X$ and a birational projective morphism $\beta\colon Y \to Z$ such that $\beta$ is the blow up of finitely many points in $Z \setminus B$ and $\gamma^{-1}(D)=\beta^{-1}(B)\sqcup\textup{Exc}\,\beta$. Moreover, one of the following holds.
\begin{enumerate}
\item The logarithmic tangent bundle $T_Z(-\textup{log}\,B)$ is numerically flat. In addition, the restriction of $\gamma$ to $Y \setminus \textup{Exc}\,\beta$ is \'etale.
\item We have $Z\cong \mathbb{P}^{n}$ and $B\cong \mathbb{P}^{n-1}$ is an hyperplane in $\mathbb{P}^{n}$. Furthermore, the restriction of $\gamma$ to $Y \setminus \gamma^{-1}(D)$ is \'etale.
\end{enumerate}
\end{thm}

In fact, a more general statement is true (see Theorem \ref{thm:main}) but its formulation is somewhat involved.

\begin{rem}
Setting and notation as in Theorem \ref{thm_intro:main}. Then $\textup{S}^n T_X(-\textup{log}\,D)\otimes\sO_X(-(K_X+D))$ is numerically flat if and only if it is nef.
\end{rem}

\subsection*{Previous results}
In \cite{jahnke_radloff_13}, Jahnke and Radloff proved that finite \'etale quotients of abelian varities are the only complex projective manifolds with numerically flat normalized tangent bundle. Greb, Kebekus, and Peternell then
generalised the theorem of Jahnke-Radloff to projective varities with klt singularities in \cite{GKP_proj_flat_JEP}. More precisely, they proved that torus quotients are the only klt varieties with semistable tangent sheaf and extremal Chern classes.
In \cite{iwai}, Iwai addressed log smooth pairs with numerically flat normalized logarithmic tangent bundle under the additional assumption that $-(K_X+D)$ is nef.

\subsection*{Outline of the proof} The general strategy of proof is similar to the one employed in \cite{jahnke_radloff_13} and \cite{GKP_proj_flat_JEP}. We had to overcome technical difficulties arising from the presence of the boundary divisor. We also had to deal with mildly singular varieties as explained below. 

The main steps for the proof are as follows.

We first run a minimal model program for the pair $(X,D)$. At each step, we contract a connected component $E$ of $D$ with $E\cong \mathbb{P}^{n-1}$ to a point. The minimal model program ends with a minimal model $(Y,B)$ of $(X,D)$ or with a quotient of 
$(\mathbb{P}^{n},H)$ by a finite cyclic group which is quasi-\'etale away from $B$. Moreover, $Y\setminus B$ has finitely many cyclic quotient singularities of type $\frac{1}{r}(1,\ldots,1)$ and the pair $(Y,B)$ is log smooth in a neighborhood of $B$. These are exactly the singularities appearing in \cite{GKP_proj_flat_JEP}. The proof of our main result relies in part on the characterization of these singularities from \cite[Proposition 4.1]{GKP_proj_flat_JEP}. 

If $(Y,B)$ is a finite cyclic quotient of $(\mathbb{P}^{n},H)$, then one easily checks that $(X,D)$ satisfies condition (2) in the statement of Theorem \ref{thm_intro:main}. 

Suppose that $(Y,B)$ is a minimal model. We show that $(X,D)$ satisfies condition (1) in the statement of Theorem \ref{thm_intro:main} as follows. 

We prove log abundance for $K_Y+B$ following the strategy employed in \cite{jahnke_radloff_13}:
we use a Shafarevich map construction to prove that it reduces to the special case where $T_Y(-\textup{log}\,B) \cong \sL^{\oplus n}$ for some rank one reflexive sheaf $\sL$.

We then consider the log Itaka fibration $f\colon Y \to Z$ of a suitable quasi-\'etale cover of $(Y,B)$. We first observe that is suffices to prove that $K_Y+B \equiv 0$. We argue by contradiction and assume that $K_Y+B \not\equiv 0$.
The divisor $K_Y+B$ cannot be big, owing to the Miyaoka-Yau inequality for minimal $\mathbb{Q}$-factorial dlt pairs proved by Guenancia and Taji in \cite{GT}. Following the ideas of \cite{jahnke_radloff_13} and \cite{GKP_proj_flat_JEP}
we will show that the log Itaka fibration of a suitable quasi-\'etale cover of $(X,D)$ is birational to an abelian
group scheme $f\colon Y_1 \to Z_1$ where $Z_1$ has finitely many log canonical singularities and $\mathbb{Q}$-ample canonical divisor. This contradicts an analogue of the Arakelov inequality for variations of Hodge structures of weight one due to Viehweg and Zuo (\cite{viehweg_zuo_arakelov}).

These steps are addressed throughout the paper, and are collected together in Section \ref{section:final}.

\subsection*{Structure of the paper} Section \ref{section:notation} gathers notation, results and global conventions that will be used throughout the paper. In Section \ref{section:nef_tangent}, we describe the structure of pairs $(X,D)$ with nef logarithmic tangent bundle. 
The proof of Theorem \ref{thm_intro:main} is long and therefore subdivided into numerous steps: Sections \ref{preparation:easy_observations} to \ref{preparation:abundance} prepare for it. With these preparations at hand, the proof of Theorem \ref{thm_intro:main} which we give in Section \ref{section:final} become
reasonably short.

\subsection*{Acknowledgements} 

The author was partially supported by the ERC project ALKAGE (ERC grant Nr 670846), the CAPES-COFECUB project Ma932/19 and the ANR project Foliage (ANR grant Nr ANR-16-CE40-0008-01).

\section{Notation, convention and used facts}\label{section:notation}

\subsection{Global conventions} Throughout the paper, all varieties are assumed to be defined over the field of complex numbers.
Given a variety $X$, we denote by $X_\textup{reg}$ its smooth locus.

\subsection{Projective space bundle}
If $\sE$ is a locally free sheaf of finite rank on a variety $X$, 
we denote by $\mathbb{P}(\sE)$ the variety $\textup{Proj}(\textup{S}^\bullet\sE)$.

\subsection{Stability}
The word \textit{semistable} will always mean \textit{slope-semistable with respect to a
given ample divisor}. We refer to \cite[Definition 1.2.12]{HuyLehn} for its precise definition.

\subsection{Reflexive hull}
Given a normal variety $X$, $m\in \mathbb{N}$, and a coherent sheaf $\sF$ on $X$, write
$\sF^{[\otimes m]}:=(\sF^{\otimes m})^{**}$, $\textup{S}^{[m]}\sF:=(\textup{S}^m\sF)^{**}$ and $\det\sF:=(\Lambda^{\textup{rank} \,\sF}\sF)^{**}$. Given any morphism $f \colon Y \to X$ of normal varieties, write 
$f^{[*]}\sF:=(f^*\sF)^{**}.$

\subsection{Singularities of pairs}

A \textit{pair} $(X,D)$ consists of a normal quasi-projective variety $X$ and an effective $\mathbb{Q}$-divisor $D$ on $X$.
A \textit{reduced pair} is a pair $(X,D)$ such that $D$ is reduced.
We will use the notions of klt and log canonical singularities for pairs
without further explanation or comment and simply refer to \cite{kollar97}
for a discussion and for their precise definitions.

\medskip

The following elementary fact will be used throughout the paper (see \cite[Proposition 3.16]{kollar97}).

\begin{fact}\label{fact:quasi_etale_cover_and_singularities}
Let $\gamma\colon Y \to X$ be a finite cover between normal complex varieties. Let $D$ be a $\mathbb{Q}$-divisor on $X$, and set
$B:=\gamma^*(K_X+D)-K_{Y}$. Suppose that $B$ is effective. Then $(X,D)$ is klt (resp. log canonical) if and only $(Y,B)$ is klt (resp. log canonical).
\end{fact}

We will also need the following definition.

\begin{defn}
A normal, quasi-projective variety $X$ is said to be of \textit{klt type} if there exists an effective $\mathbb{Q}$-divisor $D$ on $X$ such that $(X,D)$ is klt.
\end{defn}

\begin{fact}
If $X$ is of klt type and $\mathbb{Q}$-factorial, then $X$ has klt singularities.
\end{fact}

\subsection{Covering maps and quasi-\'etale morphisms} 
A \textit{cover} or \textit{covering map} is a finite and surjective morphism of normal varieties.

A morphism $\gamma\colon Y \to X$ between normal varieties is called a \textit{quasi-\'etale cover}
if $\gamma$ is finite and \'etale in codimension one. 

\medskip

The following elementary fact follows from purity of the branch locus.

\begin{fact}
Let $\gamma\colon Y \to X$ be a quasi-\'etale cover. If $D$ is a reduced and effective divisor on $X$ such that $(X,D)$ is log smooth in a Zariski open neighborhood of $D$, then $(Y,\gamma^{-1}(D))$ is log smooth in a Zariski open neighborhood of $\gamma^{-1}(D)$ as well.
\end{fact}

\subsection{Maximally quasi-\'etale varieties} Let $X$ be a normal variety. Following \cite[Paragraph 2.7]{GKP_proj_flat_JEP},
we say that $X$ is \textit{maximally quasi-\'etale} if
the natural push-forward map $\pi_1(X_\textup{reg}) \to \pi_1(X)$ induces an isomorphism between their profinite completions 
$$\wh{\pi}_1(X_\textup{reg}) \cong \wh{\pi}_1(X).$$

\begin{fact}
Let $X$ be a normal variety that is maximally quasi-\'etale. Then any finite-dimensional linear representation of $\pi_1(X_\textup{reg})$ extends to a representation of $\pi_1(X)$ by \cite[Th\'eor\`eme 1.2]{gro70}.
\end{fact}

\begin{fact}\label{fact:exitence_maximally_qe_cover}
If $X$ is any quasi-projective variety of klt type, then $X$ admits a quasi-\'etale cover 
that is maximally quasi-\'etale and of klt type by \cite[Theorem 1.14]{gkp_flat}.
\end{fact}

\subsection{Numerically flat vector bundles} We recall the definition of numerically flat vector bundles on projective varieties.

\begin{defn}
A vector bundle $\sE$ of rank $r\ge 1$ on a projective variety is called \textit{numerically flat} if $\sE$ and $\sE^*$ are nef vector bundes.
\end{defn}

\begin{rem}
Let $X$ be a projective variety and let $\sE$ be a vector bundle of rank $r \ge 1$ on $X$ with $\det \sE \cong \sO_X$. Then 
$\sE$ is numerically flat if and only if $\sE$ is nef.
\end{rem}

\begin{thm}\label{thm:flat_singular_spaces}
Let $X$ be a normal projective variety of klt type and let $\sE$ be a vector bundle of rank $r \ge 1$ on $X$. Then the following conditions are equivalent.
\begin{enumerate}
\item The vector bundle $\sE$ is numerically flat.
\item The vector bundle $\sE$ has a filtration by subbundles whose graded pieces are given by unitary representations of 
$\pi_1(X)$.
\item The vector bundle $\sE$ is flat and semistable with respect to some ample divisor.
\item The vector bundle $\sE$ is flat and semistable with respect to any ample divisor.
\end{enumerate}
\end{thm}

\begin{proof}
Let $\beta\colon Z \to X$ be a resolution of $X$. By \cite[Theorem 1.1]{takayama_fundamental_group}, the natural map 
$\pi_1(Z) \to \pi_1(X)$ is an isomorphism. Then $\textup{(1)} \Leftrightarrow \textup{(2)}$ follows easily from \cite[Theorem 1.18]{demailly_peternell_schneider94} applied to $\beta^*\sE$.

By \cite{simpson_higgs_flat} applied to $\beta^*\sE$, $\sE$ has a filtration by subbundles whose graded pieces are given by unitary representations of $\pi_1(X)$ if and only if $\beta^*\sE$ is flat and semistable with respect to any ample divisor on $Z$.
If $\beta^*\sE$ is flat and semistable with respect to any ample divisor on $Z$, then 
$\sE$ is easily seen to be flat and semistable as well using the fact that big and nef divisors are limits of ample divisors.
If $\sE$ is flat and semistable with respect to any ample divisor, then $\beta^*\sE$ is flat and 
semistable with respect to any ample divisor by \cite[Theorem 3.9]{kebekus_al_naht} applied to $\sE$ equipped with the zero Higgs field.

Finally, $\textup{(3)} \Leftrightarrow \textup{(4)}$ follows again from \cite[Theorem 3.9]{kebekus_al_naht} applied to $\sE$ equipped with the zero Higgs field.
\end{proof}

\subsection{Projectively flat sheaves} One key notion is that of a projectively flat vector bundle. We
recall the definition.

\begin{defn}\label{def:proj_flat}
A vector bundle $\sE$ of rank $r\ge 1$ on a variety $X$ is called \textit{projectively flat} if $\mathbb{P}(\sE)$ is defined by a  
representation $\pi_1(X) \to \textup{PGL}(r,\mathbb{C})$.
\end{defn}

\begin{fact}
Setup as in Definition \ref{def:proj_flat}. Suppose in addition that $X$ is smooth.
Then $\sE$ is projectively flat if and only if there exists a closed subset $Z \subset X$ of codimension $\codim Z \ge 2$ such that
the restriction $\sF|_{X\setminus Z}$ of $\sF$ to $X\setminus Z$ is projectively flat.
\end{fact}

\begin{lemma}\label{lemma:chern_classes}
Let $\sE$ be a vector bundle on a projective variety. If $\sE$ is projectively flat, then $c_i(\sE)\equiv\frac{1}{r^i}\binom{r}{i}c_1(\sE)^i$ for any $i \ge 1$. 
\end{lemma}

\begin{proof}
This follows from the arguments of \cite[Proof of Proposition 1.1 (2)]{jahnke_radloff_13} applied to the pull-back of $\sE$ to a resolution of $X$.
\end{proof}

The following result extends \cite[Theorem 1.1]{jahnke_radloff_13} and \cite[Proposition 1.1]{jahnke_radloff_13}
to projective varieties of klt type that are maximally quasi-\'etale.

\begin{thm}\label{thm:JR_singular}
Let $X$ be a normal projective variety of klt type that is maximally quasi-\'etale and let $\sE$ be a reflexive sheaf of rank $r \ge 1$ on $X$.
Then the following conditions are equivalent.
\begin{enumerate}
\item The sheaf $\sE$ is semistable with respect to some ample divisor and $\sE|_{X_\textup{reg}}$ is locally free and projectively flat.
\item The sheaf $\sE$ is semistable with respect to any ample divisor and $\sE|_{X_\textup{reg}}$ is locally free and projectively flat.
\item The sheaf $(\textup{S}^r\sE\otimes\det \sE^*)^{**}$ is locally free and numerically flat.
\end{enumerate}
\end{thm}

\begin{proof}
Suppose first that $(\textup{S}^r\sE\otimes\det \sE^*)^{**}$ is locally free and numerically flat. By Theorem \ref{thm:flat_singular_spaces}, the vector bundle $(\textup{S}^r\sE\otimes\det \sE^*)^{**}$ is then flat and semistable with respect to any ample divisor. This in turn implies that $\sE$ is semistable with respect to any ample divisor. Moreover, the sheaf 
$\sE|_{X_\textup{reg}}$ is easily seen to be locally free and projectively flat using \cite[Proposition 3.7]{GKP_proj_flat_JAG}.
This proves the implication $\textup{(3)} \Rightarrow \textup{(2)}$

The implication $\textup{(2)} \Rightarrow \textup{(1)}$ is obvious.

To finish the proof of the theorem, it remains to prove $\textup{(1)} \Rightarrow \textup{(3)}$. Suppose that $\sE$ is semistable with respect to some ample divisor and that $\sE|_{X_\textup{reg}}$ is locally free and projectively flat.
By \cite[Corollary 3.6]{GKP_proj_flat_JAG}, the sheaf $(\textup{S}^r\sE\otimes \det \sE^*)^{**}$ is locally free and flat.
Moreover, $(\textup{S}^r\sE\otimes \det \sE^*)^{**}$ is semistable by \cite[Theorem 3.1.4]{HuyLehn}. By Theorem \ref{thm:flat_singular_spaces}, we conclude that $(\textup{S}^r\sE\otimes \det \sE^*)^{**}$ is numerically flat, completing the proof of the theorem.
\end{proof}

\begin{rem}\label{remark:JR_singular}
The implication $\textup{(3)} \Rightarrow \textup{(2)}$ remains true if $X$ is only assumed to be of klt type and not necessarily maximally quasi-\'etale.
\end{rem}

We will need the following immediate consequence of \cite[Theorem 1.1]{jahnke_radloff_13} together with \cite[Proposition 1.1]{jahnke_radloff_13}.

\begin{cor}\label{cor:restriction}
Let $X$ be a normal projective variety and let $\sE$ be a reflexive sheaf of rank $r \ge 1$ on $X$.
Suppose that the sheaf $(\textup{S}^r\sE\otimes\det \sE^*)^{**}$ is locally free and numerically flat.
Let $Y \subseteq X$ be a smooth projective variety such that $\sE$ is locally free along $Y$.
Then the vector bundle $\sE|_Y$ is semistable with respect to any ample divisor on $Y$ and projectively flat.
\end{cor}

The following observation will be crucial for the proof of our main result. We refer to \cite{lu_taji_torus_quotient}
(see also \cite[Definition 4.3]{gkp_flat}) for the definition of intersection numbers of line bundles with $\mathbb{Q}$-Chern classes of reflexive sheaves.

\begin{lemma}\label{lemma:proj_flat_versus_flat}
Let $X$ be a normal projective variety of klt type that is maximally quasi-\'etale and let $\sE$ be a reflexive sheaf of rank $r \ge 1$ on $X$. Suppose that $\det(\sE)$ is $\mathbb{Q}$-Cartier.
Suppose in addition that $\sE$ is semistable with respect to some ample divisor and that $\sE|_{X_\textup{reg}}$ is locally free and projectively flat. If $c_1(\sE)\equiv 0$ then $\sE$ is locally free and flat. 
\end{lemma}

\begin{proof}
There exists a closed subset $Z \subset X$ with $\codim_X Z \ge 3$ and a quasi-\'etale $\mathbb{Q}$-structure on 
$X^\circ:=X \setminus Z$. We may also assume without loss of generality that $\sE|_{X^\circ}$ is a $\mathbb{Q}$-vector bundle.
Let $\gamma^\circ\colon Y^\circ \to X^\circ$ be a global Mumford cover. Then $\sG^\circ:=(\gamma^\circ)^{[*]}(\sE|_{X^\circ})$ is locally free and projectively flat. Set $n:=\dim X$ and let $H_1,\ldots,H_{n-2}$ be very ample divisors on $X$. Let $S$ be a general complete intersection surface of elements in $|H_1|\times \cdots \times |H_{n-2}|$ and set $T:=(\gamma^\circ)^{-1}(S)$.
Then $S \subset X^\circ$ and
$$\wh{c}_1(\sE)^2\cdot H_1 \cdots H_2=(\gamma^\circ|_T)_*c_1(\sG^\circ)^2\quad\text{and}\quad \wh{c}_2(\sE)\cdot H_1 \cdots H_2=(\gamma^\circ|_T)_*c_2(\sG^\circ).$$
Moreover, since $\det \sE$ is $\mathbb{Q}$-Cartier by assumption, we have $\wh{c}_1(\sE)\equiv c_1(\sE)\equiv 0$. 
By Lemma \ref{lemma:chern_classes} applied to $\sG^\circ|_T$, we must have 
$\wh{c}_1(\sE)^2\equiv 0$ and $\wh{c}_2(\sE)\equiv 0$. Therefore, by \cite[Theorem 1.4]{lu_taji_torus_quotient}, the sheaf 
$\sE|_{X_\textup{reg}}$ is locally free and flat. Since $X$ is maximally quasi-\'etale, we infer that $\sE$ is likewise locally free and flat. This finishes the proof of the lemma.
\end{proof}

\subsection{Logarithmic differential forms}
Let $X$ be a smooth variety of dimension $n \ge 1$, and let
$D \subset X$ a divisor with simple normal crossings. Let  
$$T_X(- \textup{log}\, D) \subseteq T_X = \Der_{\mathbb{C}}(\sO_X)$$ 
be the subsheaf consisting of those derivations that preserve the ideal sheaf $\sO_X(-D)$. One easily checks that the 
\textit{logarithmic tangent sheaf} $T_X(-\textup{log}\,D)$ is a locally free sheaf of Lie subalgebras of $T_X$, having the same restriction as $T_X$ to $X \setminus D$. If $D$ is defined at $x$ by the equation $x_1\cdots x_k=0$, where $x_1,\ldots,x_k$ form part of a regular system of parameters $(x_1,\ldots,x_n)$ of the local ring $\sO_{X,x}$ of $X$ at $x$, then a local basis of $T_X(- \textup{log}\,D)$ (after localization at $x$) consists of
$$x_1 \partial_1, \ldots, x_k \partial_k, \partial_{k+1}, \ldots, \partial_n,$$ 
where $(\partial_1, \ldots, \partial_n)$ is the local basis of $T_X$
dual to the local basis $(dx_1, \ldots, dx_n)$ of $\Omega^1_X$.

A local computation shows that $T_X(-\textup{log}\, D)$ can be identified with the subsheaf of $T_X$ containing those vector fields that are tangent to $D$ at smooth points of $D$.

The dual of $T_X(-\textup{log}\, D)$ is the sheaf $\Omega^1_X( \textup{log}\,D)$ of logarithmic differential $1$-forms. More generally,
if $1 \le p \le n$, then $\Omega^p_X( \textup{log}\,D):=\wedge^p \Omega^1_X( \textup{log}\,D)$ is the sheaf of \textit{logarithmic differential $p$-forms}, that is, of rational $p$-forms $\alpha$ on $X$ such that $\alpha$ and $d\alpha$ have at most simple poles along $D$. 
The top exterior power $\det \Omega^1_X( \textup{log}\,D)=\Omega^n_X( \textup{log}\,D)$ is the invertible sheaf $\sO_X(K_X+D)$, where $K_X$ denotes a canonical divisor.

\begin{lemma}\label{lemma:blow_up_point}
Let $(X,D)$ be a reduced log smooth pair with $X$ projective and let $\beta\colon Y \to X$ be the blow-up of $X$ at a point $x \in X \setminus D$ with exceptional divisor $E$. Then $\Omega^1_Y(\textup{log}\,(\beta^{-1}(D)+E))$ is semistable and projectively flat if and only if $\Omega^1_X(\textup{log}\,D)$ is semistable and projectively flat.
\end{lemma}

\begin{proof}
An easy local computation shows that the composition
$$\beta^*\Omega^1_X(\textup{log}\,D) \to \Omega^1_Y(\textup{log}\,(\beta^{-1}(D)) \to \Omega^1_Y(\textup{log}\,(\beta^{-1}(D)+E)$$
yields an isomorphism
$$\beta^*\Omega^1_X(\textup{log}\,D)\cong \Omega^1_Y(\textup{log}\,(\beta^{-1}(D)+E)(-E).$$
The claim now follows from Theorem \ref{thm:JR_singular}.
\end{proof}

\begin{lemma}\label{lemma:blow_up_strata}
Let $(X,D)$ be a reduced log smooth pair with $X$ projective and let $\beta\colon Y \to X$ be the blow-up of $X$ along a strata of $(X,D)$ with exceptional divisor $E$. Then $\Omega^1_Y(\textup{log}\,(\beta^{-1}(D)+E))$ is 
semistable and projectively flat if and only if $\Omega^1_X(\textup{log}\,D)$ is semistable and projectively flat.
\end{lemma}

\begin{proof}
An easy local computation shows that standard pull-back map of K\"ahler differentials
$$\beta^*\Omega^1_X(\textup{log}\,D)\to \Omega^1_Y(\textup{log}\,\beta^{-1}(D))$$
is an isomorphism.
The claim now follows from Theorem \ref{thm:JR_singular}.
\end{proof}

\subsection{Reflexive (logarithmic) differentials forms} \label{subsection:pull-back}
Given a normal variety $X$, we denote the sheaf of K\"{a}hler differentials by
$\Omega^1_X$. 
If $1 \le p \le \dim X$ is any integer, write
$\Omega_X^{[p]}:=(\Omega_X^p)^{**}$.
The tangent sheaf $(\Omega_X^1)^*$ will be denoted by $T_X$. 

Let $D$ be a reduced effective divisor on $X$.  
If $1 \le p \le \dim X$ is any integer, we write $\Omega_X^{[p]}(\textup{log}\, D)$ for
the reflexive sheaf on $X$ whose restriction to the open set $U$ where $(X,D)$ is log smooth is the sheaf of logarithmic differential $p$-forms $\Omega_U^{p}(\textup{log}\, D|_U)$. We will refer to it as the sheaf of \textit{reflexive logarithmic $p$-forms}.

The dual of $\Omega^{[1]}_X( \textup{log}\,D)$ is the \textit{logarithmic tangent sheaf} $T_X(-\textup{log}\, D)$.

\begin{lemma}[{\cite[Lemma 2.7]{druel_lo_bianco}}]\label{lemma:pull_back_cover}
Let $\gamma\colon Y \to X$ be a finite cover between normal varieties, and let $D$ be a reduced effective divisor on $X$.
Suppose that $\gamma$ is quasi-\'etale over $X \setminus D$ and set $B:=\gamma^*(K_X+D)-K_Y$. Then $B$ is reduced and effective. Moreover, the standard pull-back map of K\"ahler differentials induces an isomorphism
$$\gamma^{[*]}\Omega_X^{[1]}(\textup{log}\, D) \cong \Omega_Y^{[1]}(\textup{log}\, B).$$
\end{lemma}

\subsection{Quotient singularities} A complex algebraic variety $X$ of dimension $n\ge 2$ is said to have a quotient singularity of type $\frac{1}{r}(1,\ldots,1)$ at $x \in X$ for some positive integer $r$ if an analytic neighborhood of $x$ is biholomorphic to an analytic neighborhood of the origin in $\mathbb{C}^n/G$, where $G=<\zeta>$ is a cyclic group of order $r$ acting on $\mathbb{C}^n$ by 
$\zeta\cdot (z_1,\ldots,z_n)=(\zeta z_1,\ldots,\zeta z_n)$. Then $x$ is an isolated singularity and the blow-up $Y$ of $X$ at $x$ is a resolution of $(X,x)$. Moreover, the exceptional divisor $E$ of the blow-up is isomorphic to $\mathbb{P}^{n-1}$ and $\sN_{E/Y}\cong \sO_{\mathbb{P}^{n-1}}(-r)$. 

\medskip

The following characterization of quotient singularities of type $\frac{1}{r}(1,\ldots,1)$ is shown in \cite{GKP_proj_flat_JEP}.

\begin{prop}\label{prop:characterization_quotient_singularity}
Let $(X,x)$ be a germ of normal complex space of dimension $n \ge 2$ with log canonical singularities. Assume that the sheaf $\Omega_X^{[1]}$ of reflexive differential forms satisfies $\Omega_X^{[1]}\cong \sL^{\oplus n}$, where $\sL$ is reflexive of rank one. Then $(X,x)$ has a cyclic quotient singularity of type $\frac{1}{r}(1,\ldots,1)$ for some integer $r \ge 1$. In particular, $(X,x)$ has klt singularities.
\end{prop}

\begin{proof}
The arguments of \cite[Proof of Proposition 4.1]{GKP_proj_flat_JEP} apply verbatim. One only needs to replace the use of 
\cite[Theorem 3.8]{druel_zl} by \cite[Theorem 1.1]{druel_zl}.
\end{proof}

\begin{prop}\label{prop:characterization_quotient_singularity_2}
Let X be a normal, irreducible complex space of dimension $n \ge 2$ with log canonical singularities. Suppose that there exists a representation $\rho\colon \pi_1(X) \to \textup{PGL}(n,\mathbb{C})$ such that $\mathbb{P}(T_X|_{X_\textup{reg}})$ is defined by the induced representation $\pi_1(X_{\textup{reg}}) \to \pi_1(X) \to \textup{PGL}(n,\mathbb{C})$. Then $X$ has isolated cyclic quotient singularities of type $\frac{1}{r}(1,\ldots,1)$.
\end{prop}

\begin{proof}
This follows easily from \cite[Proposition 3.11]{GKP_proj_flat_JAG} together with Proposition \ref{prop:characterization_quotient_singularity}.
\end{proof}

The next result extends Lemma \ref{lemma:blow_up_point} to projective varieties with (isolated) cyclic quotient singularities of type $\frac{1}{r}(1,\ldots,1)$. 

\begin{lemma}\label{lemma:resolution}
Let $(X,D)$ be a reduced pair with $X$ projective of dimension $n \ge 2$. Suppose that $(X,D)$ is log smooth in a Zariski open neighborhood of $D$ and that $X \setminus D$ has only (isolated) cyclic quotient singularities of type $\frac{1}{r}(1,\ldots,1)$.
Let $\beta\colon Z \to X$ be the blow-up of the finitely many singular points with exceptional divisor $E$. Then $\Omega^1_Z(\textup{log}\,(\beta^{-1}(D)+E))$ is semistable and projectively flat if and only if the reflexive sheaf $\Omega^{[1]}_X(\textup{log}\,D)$ is semistable and 
$\Omega^{[1]}_X(\textup{log}\,D)|_{X_\textup{reg}}$ is projectively flat.
\end{lemma}

\begin{proof}

If $\Omega^1_Z(\textup{log}\,(\beta^{-1}(D)+E))$ is semistable and projectively flat, then $\Omega^{[1]}_X(\textup{log}\,D)$ is easily seen to be semistable using the fact that big and nef divisors are limits of ample divisors. Moreover, 
$\Omega^{[1]}_X(\textup{log}\,D)|_{X_\textup{reg}}$ is projectively flat.

Conversely, suppose that $\Omega^{[1]}_X(\textup{log}\,D)$ is semistable and that
$\Omega^{[1]}_X(\textup{log}\,D)|_{X_\textup{reg}}$ is projectively flat.

Suppose first that $X$ is maximally quasi-\'etale. We denote the irreducible components of $E$ by $E_1,\ldots,E_N$. By construction, $E_i=\beta^{-1}(x_i)$ for some $x_i\in X\setminus D$ and the $E_i$ are pairwise disjoint. By assumption, $X$ has a quotient singularity of type $\frac{1}{r_i}(1,\ldots,1)$ at $x_i$ for some positive integer $r_i$.
Set $m:=\prod_{1 \le i\le N}r_i$. Then $\sO_X(m(K_X+D))$ is locally free.
By Theorem \ref{thm:JR_singular}, the sheaf
$\textup{S}^{[m]}(\textup{S}^{[n]} \Omega^{[1]}_X(\textup{log}\,D))\otimes \sO_X(-m(K_X+D))$ is locally free and numerically flat.
An easy local computation now shows that there are isomorphisms of locally free sheaves
$$\beta^* \textup{S}^{[m]}(\textup{S}^{[n]}\Omega^{[1]}_X(\textup{log}\,D))\cong \textup{S}^{m}(\textup{S}^{n}\Omega^1_Z(\textup{log}\,(\beta^{-1}(D)+E)))\otimes \sO_Z\Big(\sum_{1 \le i \le N}\frac{mn}{r_i}E_i\Big)$$
and
$$\sO_Z(m(K_Z+\beta^{-1}(D)+E)) \cong \beta^*\sO_X(m(K_X+D))\otimes\sO_Z\Big(\sum_{1 \le i \le N} \frac{mn}{r_i}E_i\Big).$$
Thus, we have
$$\beta^*\big(\textup{S}^{[m]}(\textup{S}^{[n]}\Omega^{[1]}_X(\textup{log}\,D))\otimes\sO_X(-m(K_X+D))\big)\cong\textup{S}^{m}(\textup{S}^{n}\Omega^1_Z(\textup{log}\,(\beta^{-1}(D)+E)))\otimes \sO_Z(-m(K_Z+\beta^{-1}(D)+E)).$$ 
This implies that $\textup{S}^{m}(\textup{S}^{n}\Omega^1_Z(\textup{log}\,(\beta^{-1}(D)+E)))\otimes \sO_Z(-m(K_Z+\beta^{-1}(D)+E))$ is numerically flat. It follows that
$\textup{S}^{n}(\Omega^1_Z(\textup{log}\,(\beta^{-1}(D)+E)))\otimes \sO_Z(-(K_Z+\beta^{-1}(D)+E))$
is likewise numerically flat. Hence $\Omega^1_Z(\textup{log}\,(\beta^{-1}(D)+E))$ is semistable with respect to some ample divisor on $Z$ and projectively flat by Theorem \ref{thm:JR_singular}.

In the general case, by Fact \ref{fact:exitence_maximally_qe_cover}, there exists a quasi-\'etale cover $\gamma\colon Y \to X$ 
that is maximally quasi-\'etale and of klt type. Set $D_1:=\gamma^{-1}(D)$. By purity of the branch locus, $(X_1,D_1)$ is log smooth in a Zariski open neighborhood of $D_1$ and the singular locus $S_1$ of $X_1$ is contained in the preimage $\gamma^{-1}(S)$ of the singular locus $S$ of $X$. If $X$ has 
a quotient singularity of type $\frac{1}{r(x)}(1,\ldots,1)$ at $x\in X \setminus D$ and $x_1 \in \gamma^{-1}(x)$, then $X_1$ has a quotient singularity of type $\frac{1}{r(x_1)}(1,\ldots,1)$ for some $r(x_1)$ dividing $r(x)$. Moreover, 
$\sI_x\cdot\sO_{X_1,x_1}=\sI_{x_1}^{r(x)/r(x_1)}$. In particular, if $\beta_1\colon Z_1 \to X_1$ denotes the blow-up of 
$\gamma^{-1}(S)$ with exceptional divisor $E_1$, then $\gamma\circ\beta_1$ factors through $\beta$ by the universal property of blow-ups. Let $\gamma_1\colon Z_1 \to Z$ denote the induced morphism. Notice that $\gamma_1$ is finite and \'etale over $Z\setminus E$ and that $Z_1$ is smooth. By Lemma \ref{lemma:pull_back_cover}, we have
$$\gamma_1^*\Omega^1_Z(\textup{log}\,(\beta^{-1}(D)+E)) \cong \Omega^1_{Z_1}(\textup{log}\,(\beta_1^{-1}(D_1)+E_1))\quad
\textup{and}\quad \gamma^*\Omega^{[1]}_X(\textup{log}\,D)\cong \Omega^{[1]}_{X_1}(\textup{log}\,D_1).$$
The general case now follows easily from the case treated previously.
\end{proof}

\subsection{Abundance} The following special case of the abundance conjecture is an easy consequence of \cite[Theorem 1.7]{fujino17}.

\begin{prop}\label{prop:addition}
Let $(X,D)$ be a log canonical pair with $X$ projective and let $f \colon X \map Y$ be a dominant, almost proper, rational map with connected fibers onto a normal projective variety $Y$. Let $F$ be a general fiber and set $D_F:=D|_F$. Suppose that $K_F+D_F$ is abundant and that $Y$ is of general tytpe. Then $K_X+D$ is abundant.
\end{prop}

\begin{proof}
Let $\beta\colon Z \to X$ be a resolution of the indeterminacy of $f$. Replacing $Z$ by a further blow-up, we may also assume that 
$\beta$ is a log resolution of $(X,D)$. 
Write $K_Z+B=\beta^*(K_X+D)+E$ where $B$ and $E$ are effective with no common components, $\beta_*B=D$ and $E$ is $\beta$-exceptional.
Set $G:=\beta^{-1}(F)$, $B_G:=B|_G$ and $E_G:=E|_G$. 
By general choice of $F$, we may also assume that $\beta_G:=\beta|_G \colon G \to F$ is a log resolution of $(F,D_F)$, that 
$(\beta_G)_*B_G=D_F$ and that $E_G$ is $\beta_G$-exceptional. By the adjunction formula, 
$K_G+B_G=\beta_G^*(K_F+D_F)+E_G$. Then we have 
$$
\begin{array}{ccll}

\kappa(K_G+B_G) & = & \kappa(K_F+D_F) \\
& = & \nu(K_F+D_F) & \text{since $K_F+D_F$ is abundant}\\ 
& = & \nu(K_G+B_G) & \text{by \cite[Proposition V.2.7(7)]{nakayama04}}\\
\end{array}
$$
and $K_G+B_G$ is likewise abundant. Note that $(Z,B)$ is log canonical. Now we have
$$
\begin{array}{ccll}
\kappa(K_X+D) & \le & \nu(K_X+D) & \text{by \cite[Proposition V.2.7(2)]{nakayama04}}\\
& = & \nu(K_Z+B) & \text{by \cite[Proposition V.2.7(7)]{nakayama04}}\\
& \le & \nu(K_G+B_G)+\dim Y & \text{by \cite[Proposition V.2.7(9)]{nakayama04}}\\
& = & \kappa(K_G+B_G)+\dim Y & \text{since $K_G+B_G$ is abundant}\\
& = & \kappa(K_Z+B) & \text{by \cite[Theorem 1.7]{fujino17}}\\
& = & \kappa(K_X+D).
\end{array}
$$
It follows that $\kappa(K_X+D) = \nu(K_X+D)$, as claimed.
\end{proof}

\subsection{Pseudo-effective divisors} We will need the following easy consequence of the Hodge index theorem.

\begin{lemma}\label{lemma:hodge_index}
Let $X$ be a projective manifold of dimension $n \ge 2$, and let $D=\sum_{i \in I}a_i D_i$ be a $\mathbb{Q}$-divisor on $X$. Suppose that $D \cdot D_i \equiv 0$ for any $i \in I$ and that $D$ is pseudo-effective. Then there exist numbers $b_i \in \mathbb{Q}_{\ge 0}$ such that $D\equiv \sum_{i\in I}b_i D_i$.
\end{lemma}

\begin{proof}
By the Lefschetz theorem on hyperplane sections together with a theorem of Bertini (see \cite[I 6.10]{jouanolou_bertini}), we may assume without loss of generality that $n =2$. We may also assume that
the classes in $\textup{NS}(X)_\mathbb{Q}$ of the $D_i$ are $\mathbb{Q}$-linearly independent and that $a_i\neq 0$ for any $i\in I$.
Suppose also that $D\not\equiv 0$. 

Set $I_1:=\{i\in I\,|\,a_i>0\}$ and $I_2:=I \setminus I_1=\{i\in I\,|\,a_i<0\}$. If $I_2=\emptyset$ then the conclusion of Lemma \ref{lemma:hodge_index} holds true. Suppose from now on that $I_2\neq\emptyset$. Notice that $I_1\neq\emptyset$ since $D$ is pseudo-effective and $D\not\equiv 0$ by assumption. Then 
$$D^2=\big(\sum_{i\in I_1}a_iD_i\big)^2+\big(\sum_{i\in I_2}a_iD_i\big)^2+2\big(\sum_{i\in I_1}a_iD_i\big)\big(\sum_{i\in I_2}a_iD_i\big)=0\quad
\text{and}\quad \big(\sum_{i\in I_1}a_iD_i\big)\big(\sum_{i\in I_2}a_iD_i\big)\le 0.$$
Hence, either 
$(\sum_{i\in I_1}a_iD_i)^2 \ge 0$ or $(\sum_{i\in I_2}a_iD_i)^2 \ge 0$.
Notice that $\sum_{i\in I_1}a_iD_i\not\equiv 0$ and $\sum_{i\in I_2}a_iD_i\not\equiv 0$
since the classes in $\textup{NS}(X)_\mathbb{Q}$ of the $D_i$ are $\mathbb{Q}$-linearly independent and $a_i\neq 0$ for any $i\in I$. 
As $D^2=0$ and $D\cdot (\sum_{i\in I_1}a_iD_i)=D\cdot (\sum_{i\in I_2}a_iD_i)=0$, the Hodge index theorem then implies that either $D \equiv \lambda(\sum_{i\in I_1}a_iD_i)$ for some $\lambda \in \mathbb{Q}_{>0}$ or 
$D \equiv \mu(\sum_{i\in I_2}a_iD_i)$ for some $\mu \in \mathbb{Q}_{<0}$ using again the fact that $D$ is pseudo-effective. This finishes the proof of the lemma.
\end{proof}

\section{Log smooth pairs with nef logarithmic tangent bundle}\label{section:nef_tangent}

Let $(X,D)$ be a reduced log smooth pair with $X$ projective. If $T_X(-\textup{log}\, D)$ is numerically flat, then there is a smooth morphism $a\colon X \to A$ with connected fibers onto a finite \'etale quotient of an abelian variety. Moreover, the fibration $(X,D)\to A$ is locally trivial for the analytic topology and any fiber $F$ of the map $a$ is a smooth toric variety with boundary divisor $D|_F$ (see \cite[Corollary 1.7]{druel_lo_bianco}). In this section we provide a structure result for reduced log smooth pairs with nef logarithmic tangent bundles, which might be of independent interest. We refer to \cite[Theorem 3.14]{demailly_peternell_schneider94} for a somewhat related result.

\begin{prop}\label{prop:log_tangent_nef}
Let $(X,D)$ be a reduced log smooth pair with $X$ projective. Suppose that the logarithmic tangent bundle 
$T_X(-\textup{log}\, D)$ is nef.
Then there exists a smooth morphism $a \colon X \to A$ with rationally connected fibers onto a finite \'etale quotient of an abelian variety. Moreover, $D$ is a relative simple normal crossing divisor.
\end{prop}

\begin{proof}
Let $f\colon X \map R$ be the maximally rationally chain connected fibration, and let $\sG$ be the induced foliation on $X$. Set $q:=\dim R$. We may assume without loss of generality that $R$ is smooth and projective.
Recall that $f$ is an almost proper map and that its general fibers are rationally connected. 
If $q=0$, then $X$ is rationally connected and the statement holds true. 

Suppose from now on that $q \ge 1$. By \cite[Theorem 1.1]{ghs03}, $R$ is not uniruled. Thus $K_R$ is pseudo-effective by \cite[Corollary 0.3]{bdpp}. This implies that $\sG$ is given by a twisted $q$-form $\omega\in H^0(X,\Omega_X^q\otimes \sL)$ with $\sL^*$ pseudo-effective. The twisted $q$-form $\omega$ then yields an inclusion
$$\omega_D\colon \sL^* \subseteq \Omega_X^q(\textup{log}\,D).$$
On the other hand, $T_X(-\textup{log}\, D)$ is nef by assumption. This immediately implies that $\sL$ is numerically trivial. Together with \cite[Proposition 1.16]{demailly_peternell_schneider94}, this shows that $\omega_D$ is nowhere vanishing. This in turn implies that $\sG$ is regular and that any irreducible component of $D$ is generically transverse to $\sG$.
By \cite[Corollary 2.11]{split_tangent_rc}, the quotient map $a \colon X \to Y$ onto
the space of leaves of $\sG$ is then a smooth morphism between smooth projective varieties.  
Note that $\sL \cong a^*\sO_Y(-K_Y)\otimes\sO_X(-E)$ for some effective divisor $E$ on $X$. Since $K_Y$ is pseudo-effective and $\sL$ is numerically trivial, we conclude that $E=0$ and that $K_Y \equiv 0$. Moreover, the composition
$$T_X(-\textup{log}\, D) \to T_X \to a^*T_Y$$
is surjective since $\omega_D$ is nowhere vanishing.
This immediately implies that $T_Y$ is numerically flat.
By \cite[Corollary 1.19]{demailly_peternell_schneider94}, we have $c_1(T_Y)\equiv 0$ and $c_2(T_Y)\equiv 0$. As a classical consequence of Yau's theorem on the existence of a K\"ahler-Einstein metric, $X$ is then covered by a complex torus (see \cite[Chapter IV Corollary 4.15]{kobayashi_diff_geom_vb}). Finally, since the natural map $T_X(-\textup{log}\, D) \to a^*T_Y$ is surjective, we see that $D$ is a relative simple normal crossing divisor. This finishes the proof of the proposition.
\end{proof}

\begin{rem}
Setting and notation as in Proposition \ref{prop:log_tangent_nef}. If $F$ is any fiber of $a \colon X \to Y$, then 
$T_F(-\textup{log}\, D|_F)$ is nef.

\end{rem}

\section{Preparation for the proof of Theorem \ref{thm_intro:main}: easy observations}\label{preparation:easy_observations}

Throughout the present paper, we will be working in the following setup.

\begin{setup}\label{setup:main}
Let $(X,D)$ be a reduced pair with $X$ projective of dimension $n \ge 2$. Suppose that $X$ has
isolated cyclic quotient singularities of type $\frac{1}{r}(1,\ldots,1)$ and that $(X,D)$ is log smooth in a Zariski open neighborhood of $D$. Set $\sE:=\Omega^{[1]}_X(\textup{log}\,D)$.
Suppose finally that the sheaf $(\textup{S}^{n} \sE\otimes\det\sE^*)^{**}$
is locally free and numerically flat.
We denote the irreducible components of $D$ by $D=\cup_{i\in I}D_i$.
\end{setup}

\begin{rem}
Setting and notation as in \ref{setup:main}. Then the pair $(X,D)$ is log canonical and $X$ is $\mathbb{Q}$-factorial and klt.

By Remark \ref{remark:JR_singular}, the sheaf $\sE$ is semistable with respect to any ample divisor and $\sE|_{X_\textup{reg}}$ is locally free and projectively flat.

\end{rem}

\begin{rem}\label{remark:setup_maximally_qe}
Setting and notation as in \ref{setup:main}. Suppose in addition that $X$ is maximally quasi-\'etale. Then there exists a representation $\rho\colon \pi_1(X) \to \textup{PGL}(n,\mathbb{C})$ such that $\mathbb{P}(\sE|_{X_\textup{reg}})$ is defined by the induced representation $\pi_1(X_{\textup{reg}}) \to \pi_1(X) \to \textup{PGL}(n,\mathbb{C})$. If $\rho$ is the trivial representation, then $T_X(-\textup{log}\,D)\cong \sL^{\oplus n}$, where $\sL$ is a rank one reflexive sheaf.
\end{rem}

We will need the following easy observation.

\begin{lemma}\label{lemma:neagtivity_normal_bundle}
Setting and notation as in \ref{setup:main}. Let $G$ be a prime divisor on $X$ which is not contained in $D$ and let $C \subseteq G$ be a curve passing through a point in $G \setminus D$. If $(K_X+D)\cdot C \le 0$ then $G \cdot C \ge 0$.
\end{lemma}

\begin{proof}
Since $X$ has quotient singularities, there exists a positive integer $m$ such that $\textup{S}^{[m]}T_X(-\textup{log}\, D)$ and 
$\sO_X(mG)$ are locally free sheaves. Moreover, the composition
$$\textup{S}^m{T_X(-\textup{log}\, D)}|_{G\cap X_\textup{reg}} \to \textup{S}^mT_X|_{G\cap X_\textup{reg}} \to \textup{S}^m\sN_{G/X}|_{G\cap X_\textup{reg}}
\cong \sO_X(mG)|_{G\cap X_\textup{reg}}$$
yields a generically surjective morphism of locally free sheaves
$$\textup{S}^{[m]}T_X(-\textup{log}\, D)|_C \to \sO_X(mG)|_C.$$
By Corollary \ref{cor:restriction} together with \cite[Theorem 3.1.4]{HuyLehn}, the vector bundle $\textup{S}^{[m]}T_X(-\textup{log}\, D)|_C$ is semistable of non-negative degree, and hence $G \cdot C \ge 0$, proving the lemma.
\end{proof}

We will also need the following minor generalization of \cite[Lemma 5.2]{jahnke_radloff_13}.

\begin{lemma}\label{lemma:trivial_family}
Setting and notation as in \ref{setup:main}. Let $B$ be a smooth curve and let $A$ be an abelian variety of dimension $\dim A \ge 1$.
Let also $\gamma \colon B \times A \map X$ be a generically finite rational map. Suppose in addition that $\gamma$ is well defined along $\{b\}\times A$ for $b \in B$ general and that $\gamma(\{b\}\times A)\cap D=\emptyset$.
Let $C \subset B\times A$ be a general complete intersection curve of very ample divisors. Then $(K_X+D)\cdot \gamma(C) \le 0$. 
\end{lemma}

\begin{proof}Let $\beta\colon Z \to X$ be the blow-up of the finitely many singular points with exceptional divisor $E$. Then $\Omega^1_Z(\textup{log}\,(\beta^{-1}(D)+E))$ is semistable and projectively flat by Lemma \ref{lemma:resolution} together with Remark \ref{remark:JR_singular}. We denote the irreducible components of $E$ by $E=\cup_{j\in J}E_J$. Given $j\in J$, let $r_j$ be the positive integer such that $\sN_{E_j/Z}\cong \sO_{\mathbb{P}^{n-1}}(-r_j)$. An easy computation then shows that
$$ K_Z+\beta^{-1}(D)+E \equiv \beta^*(K_X+D)+\sum_{j\in J} \frac{n}{r_j}E_j.$$
Therefore, replacing $(X,D)$ by $(Z,\beta^{-1}(D)+E)$, if necessary, we may assume without loss of generality that $(X,D)$ is a log smooth pair.

Let $U \subseteq B \times A$ be an open set with complement of codimension at least $2$ such that $\gamma|_U$ is a morphism. 
Let $F \subset B \times A$ be the divisor such that $(\gamma|_U)^{-1}(D)=F\cap U$. Observe that $F$ is a finite union of fibers of the projection $\textup{pr}_{B}\colon B \times A \to B$.
Let also $C \subset B \times A$ be a general complete intersection curve of very ample divisors. Then $C \subset U$ by general choice of $C$. The standard pull-back map of K\"ahler differentials gives a generically injective morphism 
$$
(\gamma|_U)^*\Omega^{1}_X(\textup{log}\,D) \to \Omega^1_U(\textup{log}\,F|_U)
\cong ((\textup{pr}_B^*\Omega^1_B)\otimes\sO_{B\times A}(F)\oplus 
\textup{pr}_A^*\Omega^1_A)|_U.
$$
By general choice of $C$, the composition 
$$(\gamma|_C)^*\Omega^{1}_X(\textup{log}\,D) 
\to 
((\textup{pr}_B^*\Omega^1_B)\otimes\sO_{B\times A}(F)\oplus \textup{pr}_A^*\Omega^1_A)|_C
\to 
\textup{pr}_A^*\Omega^1_A|_C\cong \sO_C^{\oplus \dim A}
$$
is then generically surjective. On the other hand, by Corollary \ref{cor:restriction}, the vector bundle
$(\gamma|_C)^*\Omega^{1}_X(\textup{log}\,D)$ is semistable, and hence we must have  
$(K_X+D)\cdot \gamma(C) \le 0$. This finishes the proof of the lemma.
\end{proof}

In the setting of \ref{setup:main}, Lemma \ref{lemma:step1} below and \cite[Corollary 1.7]{druel_lo_bianco} give the description of the $D_i$.

\begin{lemma}\label{lemma:step1}
Setting and notation as in \ref{setup:main}.
Let $i \in I$. Then either $T_{D_i}(-\textup{log}\, (D-D_i)|_{D_i})$ is numerically flat
or $D_i$ is a connected component of $D$ and $D_i\cong \mathbb{P}^{n-1}$.
\end{lemma}

\begin{proof}
By Corollary \ref{cor:restriction}, for any smooth curve $C \to D_i$, the vector bundle $\sE|_C$ is semistable. On the other hand, the restriction to $D_i$ of the residue map yields a surjective morphism 
\begin{equation}\label{eq:res1}
\textup{res}_{D_i}|_{D_i}\colon \sE|_{D_i} = \Omega^{[1]}_X(\textup{log}\,D)|_{D_i} 
\twoheadrightarrow \sO_{D_i}.
\end{equation}
This immediately implies that $\deg_C (\sE)\le 0$ and that the locally free sheaf ${T_X(-\textup{log}\, D)}|_{D_i}$ is nef. In particular,
$-(K_X+D)|_{D_i}$ is also nef. By \cite[Lemma 2.13.2]{kebekus_kovacs_invent}, the kernel of 
$\textup{res}_{D_i}|_{D_i}$ identifies with $\Omega^1_{D_i}(\textup{log}\, (D-D_i)|_{D_i})$ so that we obtain a
surjective map
\begin{equation}\label{eq:res2}
{T_X(-\textup{log}\, D)}|_{D_i} \twoheadrightarrow T_{D_i}(-\textup{log}\, (D-D_i)|_{D_i}).
\end{equation}
It follows that $T_{D_i}(-\textup{log}\, (D-D_i)|_{D_i})$ is likewise nef. In particular, if $(K_X+D)|_{D_i}\equiv 0$, then 
$T_{D_i}(-\textup{log}\, (D-D_i)|_{D_i})$ is numerically flat.

Suppose from now on that $(K_X+D)|_{D_i}\not\equiv 0$. 

By Proposition \ref{prop:log_tangent_nef}, there is a smooth morphism $a_i\colon D_i \to A_i$ with rationally connected fibers onto a
finite \'etale quotient of an abelian variety. Moreover, $(D-D_i)|_{D_i}$ is a relative simple normal crossing divisor. Then the composition
$${T_X(-\textup{log}\, D)}|_{D_i}\twoheadrightarrow T_{D_i}(-\textup{log}\, (D-D_i)|_{D_i}) \into T_{D_i}\to a_i^*T_{A_i}$$
is surjective. On the other hand, recall that ${T_X(-\textup{log}\, D)}|_C$ is semistable and nef for any smooth curve $C \subseteq D_i$. As a consequence, if $\dim a_i(C)=1$, then $(K_X+D)\cdot C=0$. By \cite[Theorem 1.2]{nef_reduction}, we must have $\dim A_i=0$. 
Therefore $D_i$ is rationally connected and hence simply connected. Moreover, ${T_X(-\textup{log}\, D)}|_{D_i} \cong \sL_i^{\oplus n}$ for some line bundle $\sL_i$ on $D_i$. Then
$\sL_i$ is generated by its global sections by \eqref{eq:res1}. This in turn implies that 
$T_{D_i}(-\textup{log}\, (D-D_i)|_{D_i})$ is likewise generated by its global sections using \eqref{eq:res2}.

By \cite[Proposition 2.4.1]{brion_log}, the pair $(D_i,(D-D_i)|_{D_i})$ is homogeneous under a connected affine algebraic group $G_i$. In particular, the strata of $(D_i,(D-D_i)|_{D_i})$ are rational varieties. Moreover, they are the $G_i$-orbits by \cite[Corollary 2.1.3]{brion_log}. Recall that $-(K_X+D)|_{D_i}$ is nef and $(K_X+D)|_{D_i}\not\equiv 0$ by assumption. Let 
$C_i$ be a positive dimensional strata of $(D_i,(D-D_i)|_{D_i})$ such that $(K_X+D)|_{C_i}$ is not nef and $\dim C_i$ is minimal. Note that $C_i$ comes with a boundary divisor $B_i$ such that $\sO_{C_i}(K_{C_i}+B_i)\cong(\sL^*_i)^{\otimes n}|_{C_i}$ by the adjunction formula. Observe also that $B_i=0$ if $\dim C_i \ge 2$. 

Suppose first that $\dim C_i \ge 2$. By the cone theorem, there exists an extremal rational curve $\Gamma_i \subset C_i$ with $0< -K_{C_i}\cdot \Gamma_i\le \dim C_i +1$. It follows that $n \le n\deg_{\Gamma_i}\sL_i \le \dim C_i +1$, and hence $C_i=D_i$ and $-K_{C_i}\cdot \Gamma_i= \dim C_i +1$. From \cite[Theorem 1.1]{wisn_crelle}, we conclude that $D_i$ is a Fano manifold with Picard number $1$. Then \cite{kobayashi_ochiai} implies that 
$D_i \cong \mathbb{P}^{n-1}$ and $\sL_i\cong \sO_{\mathbb{P}^{n-1}}(1)$. Moreover $D_i$ is a connected component of $D$. 
If $\dim C_i=1$, then $-2+\deg_{C_i}B_i = \deg_{C_i}(K_{C_i}+B_i)=-n\deg_{C_i}\sL_i\le -n$. Therefore, $n=2$, $B_i=0$ and $C_i=D_i$. Moreover, $\deg_{C_i}\sL_i=1$. Since $B_i=0$, $C_i=D_i$ is a connected component of $D$.
This finishes the proof of the lemma.
\end{proof}

\begin{rem}
Setting and notation as in Lemma \ref{lemma:step1}. If $T_{D_i}(-\textup{log}\, (D-D_i)|_{D_i})$ is numerically flat, then we have $(K_X+D)|_{D_i}\equiv 0$.
\end{rem}

\section{Preparation for the proof of Theorem \ref{thm_intro:main}: reduction steps}\label{preparation:reduction_steps}

In this section we provide a number of reduction steps for the proof of our main result. We first show that the conclusion of Theorem \ref{thm_intro:main} holds in the special case where $K_X+D$ is torsion.

\begin{lemma}\label{lemma:torsion_case}
Setting and notation as in \ref{setup:main}. Suppose in addition that $K_X+D\equiv 0$. Then there exists a quasi-\'etale cover 
$\gamma\colon Y \to X$ such that the following hold. Set $B:=\gamma^*(K_X+D)-K_Y$.
\begin{enumerate}
\item The divisor $B$ is reduced and effective and the pair $(Y,B)$ is log smooth.
\item The logarithmic tangent bundle $T_Y(-\textup{log}\,B)$ is numerically flat.
\end{enumerate}
\end{lemma}

\begin{proof}
Let $\gamma\colon Y \to X$ be a quasi-\'etale cover that is maximally quasi-\'etale. Set $B:=\gamma^*(K_X+D)-K_Y$. Notice that $B$ is reduced and effective.
By purity of the branch locus, $(Y,B)$ is log smooth in a Zariski open neighborhood of $B$. We have $K_Y+B \equiv 0$ by construction. 
Set $\sG:=\Omega^{[1]}_Y(\textup{log}\,B)$. Then $\sG\cong\gamma^{[*]}\sE$ and 
$(\textup{S}^{n} \sG\otimes\det\sG^*)^{**}\cong\gamma^*\big((\textup{S}^{n} \sE\otimes\det\sE^*)^{**}\big)$.
It follows that $(\textup{S}^{n} \sG\otimes\det\sG^*)^{**}$
is locally free and numerically flat. By Theorem \ref{thm:JR_singular}, $\sG$ is semistable with respect to any ample divisor and $\sG|_{X_\textup{reg}}$ is locally free and projectively flat. Applying Lemma \ref{lemma:proj_flat_versus_flat}, we see that $\sG$ is locally free and flat. This easily implies that $T_Y$ is likewise locally free.
By the solution of the Zariski-Lipman conjecture for klt spaces (see \cite[Theorem 1.1]{druel_zl}), the pair $(Y,B)$ is log smooth. Moreover, $T_Y(-\textup{log}\,B)$ is numerically flat by Theorem \ref{thm:flat_singular_spaces}, completing the proof of the lemma.
\end{proof}

The next result will allow to reduce the proof of our main result to the special case where $K_X+D$ is nef.

\begin{lemma}\label{lemma:step2}
Setting and notation as in \ref{setup:main}. 
Suppose in addition that there exists a representation $\rho\colon \pi_1(X) \to \textup{PGL}(n,\mathbb{C})$ such that $\mathbb{P}(\sE|_{X_\textup{reg}})$ is defined by the induced representation $\pi_1(X_{\textup{reg}}) \to \pi_1(X) \to \textup{PGL}(n,\mathbb{C})$.
Then there exists a birational $(K_X+D)$-negative contraction $\phi \colon X \to Y$ whose exceptional locus is a disjoint union of connected components $D_i$ of $D$ such that $D_i\cong \mathbb{P}^{n-1}$ and $\phi$ contracts these divisors to points.
Moreover, one of the following holds. Set $B:=\phi_*D$.
\begin{enumerate}
\item The divisor $K_Y+B$ is nef and $(Y,B)$ satisfies all the conditions listed in Setup \ref{setup:main}.
\item We have $B \cong \mathbb{P}^{n-1}$ and there exists a finite cyclic cover $\gamma\colon \mathbb{P}^n \to Y$ which is quasi-\'etale over $Y\setminus B$. Moreover, $\gamma^{-1}(B)\cong \mathbb{P}^{n-1}$ is an hyperplane in $\mathbb{P}^{n}$.
\end{enumerate}
\end{lemma}

\begin{proof}
Let $R \subseteq \overline{\textup{NE}}(X)$ be an extremal ray such that $(K_X+D)\cdot R^*<0$. Let $\phi \colon X \to Y$ be the contraction of $R$ whose existence is guaranteed by \cite[Theorem 1.4]{fujino_non_vanishing}.
We will show that either $\dim Y=0$ and $Y$ satisfies condition (2) or $\phi$ contracts a connected component $D_i$ of $D$ such that $D_i\cong\mathbb{P}^{n-1}$ to a point.

Suppose first that $\dim Y =0$. Then $-(K_X+D)$ is ample. Moreover, $X$ is a $\mathbb{Q}$-Fano variety with klt singularities and Picard number one. In particular, $X$ is simply connected. It follows that $T_X(-\textup{log}\, D)\cong \sL^{\oplus n}$, where $\sL$ is a rank one reflexive sheaf. By Lemma \ref{lemma:step1}, we have $(D,\sL|_D) \cong (\mathbb{P}^{n-1},\sO_{\mathbb{P}^{n-1}}(1))$.
Let $d$ be the positive integer such that 
$\sO_X(D)|_D\cong \sO_{\mathbb{P}^{n-1}}(d)$. Since 
$\sL|_D\cong \sO_{\mathbb{P}^{n-1}}(1)$ we have $d\,c_1(\sL) \equiv D$. By \cite[Lemma 2.5]{codim_1_del_pezzo_fols}, the $\mathbb{Q}$-Cartier divisor $d\,c_1(\sL) - D$ is torsion. Replacing $X$ by the corresponding quasi-\'etale cover (see \cite[Definition 2.52]{kollar_mori}), we may assume that 
$d\,c_1(\sL) \sim_\mathbb{Z} D$ as Weil divisors. In other words, we have
$\sL^{[\otimes d]}\cong \sO_X(D)$. 
Let $\gamma\colon X_1 \to X$ be the associated cyclic cover (see \cite[Definition 2.52]{kollar_mori}) which is quasi-\'etale away from $D$ and smooth in a Zariski open neighborhood of $\gamma^{-1}(D)$. Set $D_1:=\gamma^*(K_X+D)-K_{X_1}$. By Lemma \ref{lemma:pull_back_cover}, $D_1$ is reduced and effective. Moreover, we have 
$$T_{X_1}(-\textup{log}\, D_1)\cong \gamma^{[*]}T_X(-\textup{log}\, D) \cong \gamma^{[*]}\sL^{\oplus n}.$$
By construction, $\gamma^{[*]}\sL\cong \sO_{X_1}(D_1)$ is Cartier. By the solution of the Zariski-Lipman conjecture for klt spaces (see \cite[Theorem 1.1]{druel_zl}), we conclude that $X_1$ is smooth.
Moreover, $K_{X_1}\sim_\mathbb{Z} -(n+1) D_1$. A classical result of Kobayashi and Ochiai then implies that  
$X_1 \cong \mathbb{P}^n$ and that $\gamma^{[*]}\sL \cong \sO_{\mathbb{P}^n}(1)$. Thus, $Y$ satisfies condition (2).

Suppose now that $0<\dim Y < \dim X$ and let $F$ be a general fiber of $\phi$. Note that $F$ is smooth since $X$ has isolated singularities. Note also that $F$ is a Fano manifold. In particular, $F$ is simply connected.
It follows that 
${T_X(-\textup{log}\, D)}|_F \cong \sL_F^{\oplus n}$, where $\sL_F$ is an ample line bundle. The composition
$${T_X(-\textup{log}\, D)}|_F \cong \sL_F^{\oplus n} \to T_X|_F\twoheadrightarrow \sN_{F/X}\cong \sO_F^{\oplus \dim Y}$$
is generically surjective, yielding a contradiction since $\sL_F$ is ample.

Suppose from now on that $\dim Y = \dim X$ and let $C$ be a curve such that $[C]\in R$.

If $C \subseteq D_i$ for some $i\in I$, then $D_i$ is a connected component of $D$ and $D_i \cong \mathbb{P}^{n-1}$ by Lemma \ref{lemma:step1}. Moreover, $\phi$ contracts $D_i$ to a point.

Suppose that $C \not\subseteq D_i$ for any $i \in I$. There exists a prime divisor $G$ on $X$ such that $G \cdot C <0$ 
(see \cite[Paragraph 1.42]{debarre}). In particular, $C \subseteq G$, and hence $G$ is not contained in the support of $D$.
Then Lemma \ref{lemma:neagtivity_normal_bundle} yields a contradiction.

This proves that $\phi$ contracts a connected component $D_i$ of $D$ such that $D_i\cong\mathbb{P}^{n-1}$ to a point. 
Note that $Y$ has an isolated cyclic quotient singularity of type $\frac{1}{r_1}(1,\ldots,1)$ at $\phi(D_i)$, where $r_i$ is  
the positive integer such that $\sO_X(D_i)|_{D_i}\cong \sO_{\mathbb{P}^{n-1}}(-r_i)$.

Set $B:=\phi_*(D)$ and $\sG:=\Omega^{[1]}_Y(\textup{log}\,B)$. By \cite[Theorem 1.2]{takayama_fundamental_group}, the representation $\rho$ factors through $\phi$ and $\mathbb{P}(\sG|_{Y_\textup{reg}})$ is induced by the induced representation 
$\pi_1(Y_{\textup{reg}}) \to \pi_1(Y) \to \textup{PGL}(n,\mathbb{C})$. Moreover, one easily checks that the pair $(Y,B)$ satisfies all the conditions listed in \ref{setup:main} using \cite[Theorem 1.2]{takayama_fundamental_group} again.
Therefore, a $(K_X+D)$-MMP terminates after finitely many steps and ends with a pair satisfying (1) or (2). This finishes the proof of the lemma.
\end{proof}

The following varieties were first considered by Iwai in \cite[Section 4.2]{iwai}.  

\begin{exmp}
Let $m$ and $n$ be a positive integers and let $Y$ be the weighted projective space $\mathbb{P}(1,\ldots,1,m)$ of dimension $n$. Then $Y$ is the quotient of $\mathbb{P}^n$ by the cyclic group $G=<\zeta>$ of order $m$ acting on $\mathbb{P}^n$ by 
$\zeta\cdot (z_0,\ldots,z_n)=(\zeta z_0,\ldots,\zeta z_n)$. It has a cyclic quotient singularity of type $\frac{1}{m}(1,\ldots,1)$ at $(0,\ldots,0,1)$. Then $Y$ satisfies all the conditions listed in Setup \ref{setup:main}.

Let $X$ be the blow-up of $Y$ at the singular point $(0,\ldots,0,1)$. Then 
$X \cong \mathbb{P}_{\mathbb{P}^{n-1}}(\sO_{\mathbb{P}^{n-1}}\oplus\sO_{\mathbb{P}^{n-1}}(-m))$ also satisfies the condition listed in Setup \ref{setup:main} by Lemma \ref{lemma:resolution} together with Theorem \ref{thm:JR_singular}.
\end{exmp}

\section{Preparation for the proof of Theorem \ref{thm_intro:main}: abundance}\label{preparation:abundance}

In this section we prove a special case of the abundance conjecture. More precisely, we prove that if the pair $(X,D)$ satisfies all the conditions listed in Setup \ref{setup:main} and moreover $(X,D)$ is minimal, then $K_X+D$ is semiample. We will need the following observation.

\begin{lemma}\label{lemma:vanishing_class}
Let $(X,D)$ be a reduced log smooth pair with $X$ of dimension $n\ge 2$. Suppose that $T_X(-\textup{log}\,D)\cong \sL^{\oplus n}$, where $\sL$ is invertible. Then there exist numbers $a_i \in \mathbb{C} $ such that $c_1(\sL) = \sum_{i\in I}a_i D_i \in H^1(X,\Omega^1_X)$. 
\end{lemma}

\begin{proof}
By assumption, $T_X(-\textup{log}\,D)\cong V\otimes \sL$ where $V$ is a complex vector space of dimension $n$. Let $v \in V\setminus \{0\}$ and set $\sL_v:=\mathbb{C}v \otimes \sL \subseteq V\otimes \sL\cong T_X(-\textup{log}\,D)$. If $v$ is a general point then $\sL_v$ is saturated in $T_X$ and defines a foliation $\sL_v \subseteq T_X$. By construction, $D_i$ is $\sL_v$-invariant for any $i \in I$. 

Let $v \in V$ be a general point. Denote by $\sN_v$ the normal sheaf of $\sL_v$. Let $U_v \subseteq X$ be the open set where $\sL_v$ is a regular foliation. By \cite[Corollary 3.4]{baum_bott70}, the vector bundle $\sN_v|_{U_v}$ admits an $\sL_v|_{U_v}$-connection. On the other hand, since $D$ is $\sL_v$-invariant, the line bundle
$\sN_{D/X}^*|_{U_v}\cong\sO_X(-D)|_{U_v}$ also admits an $\sL_v|_{U_v}$-connection. This in turn implies that 
$\sL_v^{\otimes n-1}|_{U_v}\cong \det (\sN_v|_{U_v}) \otimes \sO_X(-D)|_{U_v}$ likewise admits an $\sL_v|_{U_v}$-connection.
The arguments of \cite[Proof of Proposition 3.3]{baum_bott70} then show that 
$c_1(\sL|_{U_v})=c_1(\sL_v|_{U_v})\in H^1(U_v,\Omega^1_{U_v})$ lies in the image of the natural map
$H^1(U_v,\sN^*_v|_{U_v}) \to H^1(U_v,\Omega^1_{U_v})$.
It follows that the image of $c_1(\sL_v|_{U_v})\in H^1(U_v,\Omega^1_{U_v})$
in $H^1(U_v,\Omega^1_{U_v}(\textup{log}\,D|_{U_v}))$
lies in the subspace 
$$N_v:=H^1(U_v,\sN_{D,v}^*|_{U_v}) \subseteq H^1(U_v,\Omega^1_{U_v}(\textup{log}\,D|_{U_v}))$$
where $\sN_{D,v}:=T_X(-\textup{log}\,D)/\sL_v\cong \sL^{\oplus n-1}$.
Let now $v_1,\ldots,v_n$ be general elements in $V$ and set $U:=U_{v_1}\cap \cdots \cap U_{v_n}$. 
Since $c_1(\sL|_U) = c_1(\sL_{v_i}|_U)$, we find that the image of 
$c_1(\sL|_U)$ in $H^1(U,\Omega^1_U(\textup{log}\,D|_U))$
is contained in $N_{v_1} \cap \cdots \cap N_{v_n}=\{0\}$.
This easily implies that there exist complex numbers $a_i \in \mathbb{C} $ such that 
$c_1(\sL) = \sum_{i\in I}a_i D_i \in H^1(X,\Omega^1_X)$ using the fact that $\codim X \setminus U \ge 2$.
\end{proof}

Next, we consider the case where the logarithmic tangent sheaf is projectively trivial.

\begin{lemma}\label{lemma:special_case}
Setting and notation as in \ref{setup:main}. Suppose in addition that $K_X+D$ is nef and 
that $T_X(-\textup{log}\,D)\cong \sL^{\oplus n}$, where $\sL$ is a rank one reflexive sheaf. Then $K_X+D$ is abundant.
\end{lemma}

\begin{proof}
For the reader's convenience, the proof is subdivided into a number of steps. Set $\sM:=\sL^*$.

\medskip

\noindent\textit{Step 1.} 
By Lemma \ref{lemma:vanishing_class} applied to $(X_\textup{reg},D|_{X_\textup{reg}})$, 
there exist numbers $a_i \in \mathbb{C}$ such that $c_1(\sL)\equiv \sum_{i\in I}a_i D_i$.
Since $c_1(\sL)\in \textup{NS}(X)_\mathbb{Q}$ and $D_i\in \textup{NS}(X)_\mathbb{Q}$ for any $i\in I$, we may assume without loss of generality that $a_i\in\mathbb{Q}$.

We have $K_X+D \equiv -n c_1(\sL)$ and hence $c_1(\sL) \cdot D_i \equiv 0$ by Lemma \ref{lemma:step1}.
By Lemma \ref{lemma:hodge_index} applied to the pull-backs of $-\sum_{i\in I}a_i D_i$ and the $D_i$ on a resolution of $X$, we may assume that $a_i \le 0$ for any $i\in I$. Set $b_i:=-a_i\in \mathbb{Q}_{\ge 0}$ and $B:=\sum_{i\in I}b_i D_i$. 
Since $B\cdot D_i \equiv 0$ for any index $i \in I$, we must have $B^2\equiv 0$.

\medskip

\noindent\textit{Step 2.} Let $N$ be a sufficiently large and divisible positive integer such that $Nb_i\in\mathbb{Z}_{\ge 0}$ for any $i \in I$ and such that $\sM^{[\otimes N]}$ is Cartier. 
Suppose in addition that $\sM^{[\otimes N]} \otimes \sO_X(-NB) \in \textup{Pic}^0(X)$.
Recall that $\textup{Pic}^0(X)$ is an abelian variety since $X$ has rational singularities. Hence, we may also assume that there exists $\sP\in\textup{Pic}^0(X)$ such that 
$(\sM\otimes\sP)^{[\otimes N]} \cong \sO_X(NB)$.

Let $\gamma\colon X_1 \to X$ be the associated cyclic cover (see \cite[Definition 2.52]{kollar_mori}). Then $\gamma$
is quasi-\'etale away from $D$ and $\gamma^*B$ is an integral divisor on $X_1$ by construction. Moreover,
$\gamma^{[*]}(\sM\otimes\sP)\cong (\gamma^{[*]}\sM)\otimes\gamma^*\sP\cong \sO_X(\gamma^*B)$. 
This in turn implies that $\gamma^{[*]}\sM$ is invertible. 

Set $D_1:=\gamma^*(K_X+D)-K_{X_1}$. By Lemma \ref{lemma:pull_back_cover}, $D_1$ is reduced and effective and 
$$\Omega^{[1]}_{X_1}(\textup{log}\,D_1)\cong \gamma^{[*]}\Omega^{[1]}_X(\textup{log}\,D)\cong (\gamma^{[*]}\sM)^{\oplus n}.$$ 
In particular, $\Omega_{X_1}^{[1]}(\textup{log}\,D_1)$ is locally free. Notice that $(X_1,D_1)$ is log canonical by \cite[Proposition 3.16]{kollar97}.
By \cite[Lemma 2.10]{druel_lo_bianco}, there exists a log resolution 
$\beta\colon X_2 \to X_1$ such that 
$$\Omega^{1}_{X_2}(\textup{log}\,D_2)\cong \beta^*\Omega^{[1]}_{X_1}(\textup{log}\,D_1)\cong (\beta^*(\gamma^{[*]}\sM))^{\oplus n}$$
where $D_2$ is the largest reduced divisor contained in $\beta^{-1}(\textup{Supp}\,D_1)$.
Therefore, replacing $X$ by $X_2$ if necessary, we may assume without loss of generality that $(X,D)$ is log smooth and that 
$$\sM\otimes\sP\cong \sO_X(B)$$ with $B=\sum_{i\in I}b_iD_i$ for some $b_i\in\mathbb{Z}_{\ge 0}$. 
By \cite[Corollary 3.2]{CKP_numerical}, we have $\kappa(K_X+D) \ge 0$. In particular, if $B=0$, then $\kappa(K_X+D)=\nu(K_X+D)=0$. Suppose from now on that $B \neq 0$.

\medskip

\noindent\textit{Step 3.} 
Note that $\sO_X(K_X+D)\cong \sM^{\otimes n}$ so that $\sM$ is a nef line bundle. Recall from  \cite[Proposition 2.2]{kawamata} that $\kappa(K_X+D) \le \nu(K_X+D)$. Let $S \subseteq X$ be a general complete intersection surface of very ample divisors. Then $\sN_{S/X}$ is an ample vector bundle unless $\dim X=2$. Consider the short exact sequence (see \cite[Lemma 3.2]{druel_zl})
$$0 \to \sN_{S/X}^* \to \Omega^1_X(\textup{log}\,D)|_S\cong(\sM|_S)^{\oplus n} \to \Omega^1_S(\textup{log}\,D|_S) \to 0.$$
Since $\sM$ is nef, any composition
$$\sM|_S \into \Omega^1_X(\textup{log}\,D)|_S\cong(\sM|_S)^{\oplus n} \to \Omega^1_S(\textup{log}\,D|_S)$$
is nonzero. The Bogomolov-Sommese vansihing theorem (\cite[Corollary 6.9]{esnault_viehweg}) then implies that 
$\sM|_S$ is not big. This shows that $c_1(\sM)^2 \equiv 0$, and hence $\nu(K_X+D)\le 1$.

\medskip

\noindent\textit{Step 4.} By assumption, $T_X(-\textup{log}\,D)\cong V\otimes \sL$ where $V$ is a complex vector space of dimension $n$. Let $v \in V\setminus \{0\}$ and set $\sL_v:=\mathbb{C}v \otimes \sL \subseteq V\otimes \sL\cong T_X(-\textup{log}\,D)$. 

Recall from Step 2 that $\kappa(K_X+D)\ge 0$. Therefore, there exists an effective $\mathbb{Q}$-divisor $Q:=\sum_{j \in J}m_j Q_j$ such that 
$c_1(\sM)\sim_\mathbb{Q}Q$. Since $\sM$ is nef and $c_1(\sM)^2\equiv 0$, we must have $Q^2\equiv 0$ and $Q \cdot Q_j\equiv 0$ for any index $j\in J$. Notice that $Q \neq 0$ since $B \neq 0$ by assumption.

Suppose that $q(X)>0$. Let $a\colon X \to A$ be the Albanese morphism and let $f \colon X \to Y$ be its Stein factorization. By assumption, we have $\dim Y \ge 1$. 

Let $j \in J$ and suppose that $\dim f(Q_j) \ge 1$. Then there exists a global $1$-form on $X$ whose restriction to $\sL_v$ is nonzero at a general point in $Q_j$ (for $v \in V$ general enough). This easily implies that $\kappa(\sM) \ge 1$, and hence 
$\kappa(K_X+D)=\nu(K_X+D)=1$.

Suppose now that $\dim f(Q_j) = 0$ for any $j \in J$. By the Negativity Lemma, we must have $\dim Y=1$ since $Q^2\equiv 0$ and $Q \neq 0$. Moreover, $Q\sim_\mathbb{Q}\sum_{i\in I}r_i F_i$ where $r_i \in \mathbb{Q}_{> 0}$ and $F_i$ is a fiber of $f$ by 
Zariski's Lemma. This again implies that $\kappa(\sM) \ge 1$ and hence $\kappa(K_X+D)=\nu(K_X+D)=1$.

Suppose finally that $q(X)=0$. Then $\sM\cong\sO_X(B)$. Recall that $B\cdot D_i\equiv 0$ for any $i\in I$. It follows that $\textup{Supp}\,B$ is a union of connected components of  $D$. Set $J:=\{i\in I\,|\,b_i\neq 0\}$. 
Set also $C:=D \setminus \textup{Supp}\,B$ and set $U:=X \setminus C$. Notice that $B_j\subset U$ for any $j\in J$. 
Moreover, we have an isomorphism
$$H^0(X,\Omega^{1}_X(\textup{log}\,C))\cong H^0(X,\Omega^{1}_X(\textup{log}\,D))$$
since $\Omega^{1}_X(\textup{log}\,D)\cong \sO_X(B)^{\oplus n}$. If $h^0(X,\sM) \ge 2$, then $\kappa(\sM) \ge 1$ and hence $\kappa(K_X+D)=\nu(K_X+D)=1$. Suppose from now on that
$h^0(X,\sM) = 1$. In particular, $h^0\big(X,\Omega^{1}_X(\textup{log}\,C))=n$. Let 
$$a\colon U \to G$$ 
be the universal morphism to a semi-abelian variety (see \cite{serre_albanese} and \cite{iitaka}). Since 
$\Omega^1_X( \textup{log}\,C)$ is generated by its global sections over $X\setminus D \subseteq U$, the tangent map
$$Ta \colon T_U \to a^*T_G \cong H^0\big(X,\Omega^{1}_X(\textup{log}\,C)\big)^* \otimes \sO_U$$
is injective over $X\setminus D$. This implies that the map $a$ is generically finite.
On the other hand, since $q(X)=0$, $G\cong \mathbb{G}_m^n$ is affine. It follows that $a$ contracts $\textup{Supp}(B)$ to points. But this contradicts the fact that $B^2\equiv 0$. This finishes the proof of the proposition.
\end{proof}

The following is the main result of this section.

\begin{lemma}\label{lemma:step3}
Setting and notation as in \ref{setup:main}. Suppose in addition that $K_X+D$ is nef.
Then $K_X+D$ is semiample.
\end{lemma}

\begin{proof}By Lemma \ref{lemma:step1}, for any index $i \in I$, $(K_X+D)|_{D_i}=K_{D_i}+(D-D_i)|_{D_i}\equiv 0$.
Moreover, the pair $(D_i,(D-D_i)|_{D_i})$ is log smooth. Then \cite[Corollary 3.2]{CKP_numerical} implies that 
$(K_X+D)|_{D_i}$ is torsion. In particular, the restriction of $K_X+D$ to any log canonical center of $(X,D)$ is abundant.
As a consequence of \cite[Theorem 1.6]{fujino_gongyo_14}, in order to prove the lemma, it suffices to show that $K_X+D$ is abundant.

Let $\gamma\colon Y \to X$ be a quasi-\'etale cover that is maximally quasi-\'etale. Set $B:=\gamma^*(K_X+D)-K_Y$. Notice that $B$ is reduced and effective.
By purity of the branch locus, $(Y,B)$ is log smooth in a Zariski open neighborhood of $B$. Moreover, $Y$ has isolated cyclic quotient singularities of type $\frac{1}{r}(1,\ldots,1)$. One then readily checks that $(Y,B)$ satisfies all the conditions listed in Setup \ref{setup:main}. By \cite[Theorem 5.13]{uenoLN439}, we have $\kappa(K_X+D)=\kappa(K_Y+B)$. On the other hand, we have $\nu(K_X+D)=\nu(K_Y+B)$. Replacing $(X,D)$ by $(Y,B)$, if necessary, we may assume that $X$ is maximally quasi-\'etale. Hence, there exists a representation $\rho\colon \pi_1(X) \to \textup{PGL}(n,\mathbb{C})$ such that $\mathbb{P}(\sE|_{X_\textup{reg}})$ is defined by the induced representation $\pi_1(X_{\textup{reg}}) \to \pi_1(X) \to \textup{PGL}(n,\mathbb{C})$.

Let $\textup{H} \subseteq \textup{PGL}(n,\mathbb{C})$ be the Zariski closure of $\rho(\pi_1(X))$. This is a linear algebraic group which has finitely many connected components. Applying Selberg's Lemma and passing to an appropriate finite quasi-\'etale cover of $X$, we may assume without loss of generality that $\textup{H}$ is connected and that the image of the induced representation
$$\rho_1 \colon \pi_1(X) \to \textup{H} \to \textup{H}/\textup{Rad}(\textup{H})$$
is torsion free, where $\textup{Rad}(\textup{H})$ denotes the radical of $\textup{H}$. Let 
$$\textup{sha}_{\rho_1}\colon X \dashrightarrow Y$$
be the $\rho_1$-Shafarevich map. Recall that the rational map $\textup{sha}_{\rho_1}$ is almost proper.
By \cite[Th\'eor\`eme 1]{CCE15}, we may assume without loss of generality that 
$Y$ is a smooth projective variety of general type and that $\rho_1$ factors through $\textup{sha}_{\rho_1}$. 
Let $F$ be a general fiber of $\textup{sha}_{\rho_1}$ and set $D_F:=D|_F$. Note that $F$ is klt and that $(F,D_F)$ is log canonical.
By Proposition \ref{prop:addition}, in order to prove that $K_X+D$ is abundant, it suffices to prove that $K_F+D_F$ is abundant. In particular, we may assume from now on that $\dim Y < \dim X$.

Suppose first that either $q(F)>0$ or $q(F)=0$ and $\dim Y>0$. 

Let $a \colon F \to \textup{A}$ be the Albanese morphism, that is, the universal morphism to an abelian variety (see \cite{serre_albanese}). Since $F$ has rational singularities, we have $\dim \textup{A} = q(F)$ by 
\cite[Lemma 8.1]{kawamata85}. Let $G$ be a general fiber of the Stein factorization of $F \to a(F)$. Notice that $G$ is smooth and contained in $X_\textup{reg}$ since $X$ has isolated singularities. By \cite[Theorem 1.1]{hu_log_abundance}, in order to prove that $K_F+D_F$ is abundant, it suffices to prove that $K_G+D_G$ is abundant, where $D_G:=D|_G$. In particular, we may also assume that $\dim G >0$. Since $\rho(\pi_1(F)) \subseteq \textup{Rad}(\textup{H})$, the restriction of $\rho$ to a finite index subgroup of $\pi_1(F)$ factorizes through $a$. Hence, there exists a finite \'etale cover $G_1 \to G$ such that 
$\rho(\pi_1(G_1))$ is trivial. It follows that
$$\Omega^1_X(\textup{log}\,D)|_{G_1}=\Omega^{[1]}_X(\textup{log}\,D)|_{G_1}\cong \sM_{G_1}^{\oplus n}$$
for some line bundle $\sM_{G_1}$ on $G_1$. Next, consider the exact sequence (see \cite[Lemma 3.2]{druel_zl})
\begin{equation*}\label{conormal_sequence_G}
0 \to \sN^*_{G_1/X}\to \Omega^1_X(\textup{log}\,D)|_{G_1}\cong \sM_{G_1}^{\oplus n} \to \Omega^1_{G_1}(\textup{log}\,D_{G_1}) \to 0.
\end{equation*}
Note that $\sN_{G/X}\cong \sO_G^{\oplus \dim X - \dim G}$ since $G$ is a (general) fiber of the relative Albanese map. This immediately implies that $\sN_{G_1/X}\cong \sO_{G_1}^{\oplus \dim X - \dim G}$. As a consequence, the line bundle 
$\sM_{G_1}$ is generated by its global section. It follows that $\Omega^1_{G_1}(\textup{log}\,D_{G_1})$ and hence 
$\sO_{G_1}(K_{G_1}+D_{G_1})$ are likewise generated by their global sections. This easily implies that $K_G+D_G$ is abundant.

Suppose finally that $q(X)=0$ and $\dim Y = 0$. Then $\rho(\pi_1(X))$ is finite and
$K_X+D$ is easily seen to be abundant using Lemma \ref{lemma:special_case}. This finishes the proof of Lemma \ref{lemma:step3}.
\end{proof}

\section{Proof of Theorem \ref{thm_intro:main}}\label{section:final}

In this section we finally prove our main result. Note that Theorem \ref{thm_intro:main} is an immediate consequence of Theorem
\ref{thm:main} below together with Theorem \ref{thm:JR_singular}.

\begin{thm}\label{thm:main}
Let $(X,D)$ be a reduced pair with $X$ projective of dimension $n \ge 2$. Suppose that $X$ is of klt type and that 
$(X,D)$ is log smooth in a Zariski open neighborhood of $D$.
Suppose in addition that $T_X(-\textup{log}\,D)$ is semistable with respect to some ample divisor on $X$ and that
$T_X(-\textup{log}\,D)|_{X_\textup{reg}}$ is projectively flat.
Then there exist a smooth projective variety $Y$ and a log smooth reduced pair $(Z,B)$ as a well as a finite cover $\gamma\colon Y \to X$ and a birational projective morphism $\beta\colon Y \to Z$ such that $\beta$ is the blow up of finitely many points in $Z \setminus B$ and $\gamma^{-1}(D)=\beta^{-1}(B)\sqcup\textup{Exc}\,\beta$. Moreover, one of the following holds.
\begin{enumerate}
\item The logarithmic tangent bundle $T_Z(-\textup{log}\,B)$ is numerically flat. In addition, the restriction of $\gamma$ to $Y \setminus \textup{Exc}\,\beta$ is quasi-\'etale.
\item We have $Z\cong \mathbb{P}^{n}$ and $B\cong \mathbb{P}^{n-1}$ is an hyperplane in $\mathbb{P}^{n}$. Furthermore, the restriction of $\gamma$ to $Y \setminus \gamma^{-1}(D)$ is quasi-\'etale.
\end{enumerate}
\end{thm}

\begin{proof}
For the reader's convenience, the proof is subdivided into a number of steps. We denote the irreducible components of $D$ by $D=\cup_{i\in I}D_i$. Notice that the pair $(X,D)$ is log canonical. 
\medskip

\noindent\textit{Step 1.} Replacing $X$ by a quasi-\'etale cover, we may assume without loss of generality that $X$ is maximally quasi-\'etale. Set $\sE:=\Omega^{[1]}_X(\textup{log}\,D)$.
Then there exists a representation $\rho\colon \pi_1(X) \to \textup{PGL}(n,\mathbb{C})$ such that $\mathbb{P}(\sE|_{X_\textup{reg}})$ is defined by the induced representation $\pi_1(X_{\textup{reg}}) \to \pi_1(X) \to \textup{PGL}(n,\mathbb{C})$.
By Proposition \ref{prop:characterization_quotient_singularity_2}, $X$ has only isolated cyclic quotient singularities of type 
$\frac{1}{r}(1,\ldots,1)$. Moreover, by Theorem \ref{thm:JR_singular}, the sheaf $(\textup{S}^{n} \sE\otimes\det\sE^*)^{**}$
is locally free and numerically flat. The pair $(X,D)$ therefore satisfies all the conditions listed in Setup \ref{setup:main}.
In particular, $X$ is $\mathbb{Q}$-factorial and klt.

\medskip

\noindent\textit{Step 2.} By Lemma \ref{lemma:step2}, there exists a birational $(K_X+D)$-negative contraction $\phi \colon X \to Y$ whose exceptional locus is a disjoint union of connected components $D_i$ of $D$ such that 
$D_i\cong \mathbb{P}^{n-1}$ and $\phi$ contracts these divisors to points. 
Moreover, one of the following holds.
Set $B:=\phi_*D$. 
\begin{enumerate}
\item The divisor $K_Y+B$ is nef and $(Y,B)$ satisfies all the conditions listed in Setup \ref{setup:main}.
\item We have $B \cong \mathbb{P}^{n-1}$ and there exists a finite cyclic cover $\gamma\colon Y_1:=\mathbb{P}^n \to Y$ which is quasi-\'etale over $Y \setminus B$. Moreover, $B_1:=\gamma^{-1}(B)\cong \mathbb{P}^{n-1}$ is an hyperplane in $\mathbb{P}^{n}$.
\end{enumerate}

Suppose first that $(Y,B)$ satisfies (2). Let $X_1$ be the normalization of the fiber product $X\times_Y Y_1$ with natural morphisms $\gamma_1\colon X_1 \to X$ and $\phi_1\colon X_1 \to Y_1$. 
By construction, $\gamma_1$ is quasi-\'etale over $X\setminus D$. 
Recall that $Y$ has isolated cyclic quotient singularities at any point $y$ in $\phi(\textup{Exc}\,\phi)$. It follows that the restriction of $\gamma_1\colon X_1 \to X$ to a sufficiently small neighborhood $U$ of $\phi^{-1}(y)$ is a disjoint union of finite cyclic cover of $(U,\phi^{-1}(y))$. This in turn implies that $(X_1,\gamma_1^{-1}(\textup{Exc}\,\phi))$ is log smooth in a neighborhood of $\gamma_1^{-1}(\textup{Exc}\,\phi)$. A theorem of Moishezon (see also \cite[Theorem 2]{luo}) then applies to show that $\phi_1\colon X_1 \to Y_1$ is the blow-up of $Y_1$ at finitely many points in $Y_1 \setminus B_1$. It follows that the conclusion of Theorem \ref{thm:main} (2) holds for $(X,D)$.

Suppose now that $K_Y+B$ is nef and that the pair $(Y,B)$ satisfies all the conditions listed in Setup \ref{setup:main}.
Then $K_Y+B$ is semiample by Lemma \ref{lemma:step3}. 

Suppose in addition that $K_Y+B\equiv 0$. By Lemma \ref{lemma:torsion_case}, there exists a quasi-\'etale cover $\gamma\colon Y_1 \to Y$ such that the following properties hold in addition. Set $B_1:=\gamma^*(K_Y+B)-K_{Y_1}$ and notice that $B_1$ is reduced and effective. Then $(Y_1,B_1)$ is log smooth and $\Omega^{1}_{Y_1}(\textup{log}\,B_1)$ is numerically flat.
Let $X_1$ be the normalization of the fiber product $X\times_Y Y_1$ with natural morphisms $\gamma_1\colon X_1 \to X$ and $\phi_1\colon X_1 \to Y_1$. Set also $D_1:=\gamma_1^*(K_X+D)-K_{X_1}$. By construction, $\gamma_1$ is quasi-\'etale over $X\setminus \textup{Exc}\,\phi$. Arguing as above, we conclude that $X_1$ is the blow-up of $Y_1$ at finitely many points. 
It follows that the conclusion of Theorem \ref{thm:main} (1) holds for $(X,D)$. 


Therefore, to prove Theorem \ref{thm:main}, it suffices to show that $K_X+D \equiv 0$ under the additional assumption that 
$K_X+D$ is semiample. 

\medskip

Suppose from now on that $K_X+D$ is semiample and let $f \colon X \to Y$ be the fibration defined by $|m(K_X+D)|$ for $m \ge 1$ sufficiently divisible. We argue by contradiction and assume that $\dim Y \ge 1$.

\medskip

\noindent\textit{Step 3.} Suppose first that $\dim Y = \dim X \ge 2$. Then $K_X+D$ is nef and big. By \cite[Theorem B]{GT}, the Miyaoka-Yau inequality 
$$c_2(X,D)\cdot (-c_1(X,D))^{n-2} \ge \frac{n}{2(n+1)}c_1(X,D)^2\cdot (-c_1(X,D))^{n-2}$$
holds where $c_1(X,D)\equiv -(K_X+D)$ and $c_2(X,D):=\wh{c}_2(T_X(-\textup{log}\,D))$.
On the other hand, using Lemma \ref{lemma:chern_classes}, one easily checks that 
$c_2(X,D) = \frac{n-1}{2n}c_1(X,D)^2$, yielding a contradiction since $c_1(X,D)^n>0$ by assumption.

\medskip

Suppose from now on that $1 \le \dim Y < \dim X$. Following the ideas of \cite{jahnke_radloff_13} and \cite{GKP_proj_flat_JEP}
we will show that the log Itaka fibration of a suitable quasi-\'etale cover of $(X,D)$ is birational to an abelian
group scheme.

\medskip

\noindent\textit{Step 4.} By Lemma \ref{lemma:step1}, $(K_X+D)|_{D_i}\equiv 0$ for any index $i \in I$, and hence
$f$ maps $D_i$ to a point in $Y$. Let $F$ be a general fiber of $f$. Notice that $F$ is smooth since $X$ has only isolated singularities. Moreover $\Omega^{[1]}_X(\textup{log}\,D)|_F\cong \Omega^{[1]}_X|_F$ is projectively flat 
and semistable with respect to an ample divisor on $F$ by Corollary \ref{cor:restriction}. 
Since $c_1\big(\Omega^{[1]}_X|_F\big)\equiv 0$, we conclude that the locally free sheaf $\Omega^{[1]}_X|_F$ is numerically flat by
Lemma \ref{lemma:proj_flat_versus_flat}. The short exact sequence
$$0\to \sN_{F/X}^*\cong \sO_F^{\oplus \dim F} \to \Omega^{[1]}_X|_F \to \Omega^1_F \to 0$$
then implies that $\Omega_F^1$ is numerically flat as well.  As a classical consequence of Yau's theorem on the existence of a K\"ahler-Einstein metric, $X$ is then covered by an abelian variety (see \cite[Chapter IV Corollary 4.15]{kobayashi_diff_geom_vb}).

The arguments of \cite[Claim 5.30]{GKP_proj_flat_JEP} apply verbatim to show that, up to replacing $X$ by a finite \'etale cover, there is a commutative diagram of normal projective varieties as follows:
\begin{center}
\begin{tikzcd}[row sep=large, column sep=huge]
X \ar[d, "f"'] & X_1 \ar[l, dashed, "{\text{birational}}"'] \ar[d, "{f_1}"'] & X_2 \ar[d, "{f_2}"']\ar[l, "{\text{\'etale}}"'] \\
 Y  & Y_1 \ar[l, "{\text{birational}}"'] & Y_2 \ar[l, "{\text{\'etale}}"']\ar[ll, bend left=20, "{\gamma_2}"] 
\end{tikzcd}
\end{center}
where $f_1$ and $f_2$ are abelian schemes and $f_2$ has a level three structure. Set $Y^\circ:=Y \setminus f(D)$ and $X^\circ:=f^{-1}(Y^\circ)$. Note that $Y^\circ$ is of klt type by \cite[Theorem 7.1]{fujino_gongyo_canonical}. The arguments of \cite[Claim 5.31]{GKP_proj_flat_JEP} then show that we have a commutative diagram of projective morphisms:
\begin{center}
\begin{tikzcd}[row sep=large, column sep=huge]
X^\circ \ar[d, "{f^\circ:=f|_{X^\circ}}"'] & X_3^\circ  \ar[d, "{f_3^\circ}"']\ar[l, "{\text{\'etale}}"']\ar[r, dashed, "{\text{birational}}"] & X_4^\circ \ar[d, "{f_4^\circ}"'] & X_2^\circ \ar[d, "{f_2|_{X_2^\circ}}"']\ar[l] \\
 Y^\circ  & Y_3^\circ \ar[r, equal]\ar[l, "{\text{\'etale}}"'] & Y_4^\circ & Y_2^\circ \ar[l, "{\text{birational}}"'] \ar[lll, bend left=20, "{\gamma_2|_{Y_2^\circ}}"] 
\end{tikzcd}
\end{center}
where $f_4^\circ$ is an abelian scheme equipped with a level three structure. Moreover, the restriction of $\gamma_2$ to 
$Y_2^\circ:=\gamma_2^{-1}(Y^\circ)$ factors through the \'etale morphim $ Y_3^\circ \to Y^\circ$ and 
the induced morphism $Y_2^\circ \to Y_3^\circ$ is birational. In addition, the abelian scheme 
$X_2^\circ \to Y_2^\circ$ is the pull-back of $X_4^\circ \to Y_4^\circ$, where $X_2^\circ:=f_2^{-1}(Y_2^\circ)$.

\begin{claim}\label{claim:smooth}
The variety $Y^\circ$ is smooth and the morphism $f^\circ:=f|_{X^\circ}$ is smooth as well. In particular, $X^\circ$ is smooth. Moreover, the birational map $X_3^\circ \map X_4^\circ$ is an isomorphism.
\end{claim}

\begin{proof}
Using \cite[Corollary 1.5]{hacon_mckernan}, we see that 
the rational map $X_3^\circ \map X_4^\circ$ is a morphism. Observe that $X_3^\circ$ and $X_4^\circ$ are of klt type. 
Proposition \ref{prop:characterization_quotient_singularity_2} together with 
\cite[Theorem 1.1]{takayama_fundamental_group} then imply that $X_4^\circ$ has only isolated singularities.
Since $\dim X_4^\circ > \dim Y_4^\circ$ by assumption, we conclude that $Y_4^\circ$ and hence $X_4^\circ$ are smooth. By construction, the projective birational morphism $X_3^\circ \to X_4^\circ$ is crepant, and thus an isomorphism since $X_4^\circ$ is smooth. As a consequence, $f^\circ$ is a smooth morphism between smooth varieties, proving the claim.
\end{proof}

\medskip

\noindent\textit{Step 5.} We will need the following observation.

\begin{claim}\label{claim:codimension_one_comp_fiber}
Let $G$ be a prime divisor on $X$ such that $f(G)=f(D_i)$ for some $i \in I$.
Then there exists $j\in I$ such that $G=D_j$. 
\end{claim}

\begin{proof}
We argue by contradiction and assume that $G \neq D_j$ for any index $j \in I$. Notice that there is a curve $C \subseteq G$ passing through a general point of $G$ such that $G\cdot C<0$. This follows from Zariski's Lemma if $\dim Y=1$ and from the 
Negativity Lemma if $\dim Y \ge 2$. But this contradicts Lemma \ref{lemma:neagtivity_normal_bundle}.
\end{proof}

Let $Y_2 \to Y_5 \to Y$ be the Stein factorization of $\gamma_2$, and let $X_5$ be the normalization of the fiber product $Y_5\times_Y X$ with natural morphisms $f_5 \colon X_5 \to Y_5$ and $\gamma_5\colon X_5 \to Y_5$. Set $D_5:=\gamma_5^*(K_X+D)-K_{X_5}$.
By Step 4, the morphism $Y_5 \to Y$ is \'etale over $Y^\circ$. It then follows from Claim \ref{claim:codimension_one_comp_fiber} that
$\gamma_5$ is quasi-\'etale away from $D$. By Lemma \ref{lemma:pull_back_cover}, $D_5$ is reduced and effective and 
$\Omega^{[1]}_{X_5}(\textup{log}\,D_5)\cong \gamma_5^{[*]}\Omega^{[1]}_X(\textup{log}\,D)$. In particular, 
$\Omega^{[1]}_{X_5}(\textup{log}\,D_5)$ is locally free in a Zariski open neighborhood of $D_5$. Notice that $(X_5,D_5)$ is log canonical by \cite[Proposition 3.16]{kollar97}. By \cite[Lemma 2.10]{druel_lo_bianco}, there exists a partial log resolution 
$\beta_5\colon X_6 \to X_5$ which induces an isomorphism over $\gamma_5^{-1}(X_5\setminus D_5)$ and such that 
$$\Omega^{[1]}_{X_6}(\textup{log}\,D_6)\cong \beta_5^*\Omega^{[1]}_{X_5}(\textup{log}\,D_5)$$
where $D_6$ is the largest reduced divisor contained in $\beta_5^{-1}(\textup{Supp}\,D_5)$. 
Set $\sE_6:=\Omega^{[1]}_{X_6}(\textup{log}\,D_6)$. One then easily checks that 
$$(\textup{S}^r\sE_5\otimes\det \sE_5^*)^{**} \cong (\gamma_5\circ\beta_5)^*(\textup{S}^r\sE\otimes\det \sE^*)^{**}.$$
It follows that the sheaf $(\textup{S}^r\sE_5\otimes\det \sE_5^*)^{**}$ is locally free and numerically flat.
Moreover, the restriction of $\mathbb{P}(\sE_5)$ to the smooth locus of $X_5$ is defined by the 
representation induced by $\rho$. We obtain a commutative diagram as follows:
\begin{center}
\begin{tikzcd}[row sep=large, column sep=huge]
X \ar[d, "f"'] & & X_6 \ar[ll, "{\text{generically finite}}"'] \ar[d, "{f_6}"']\ar[r, dashed, "{\text{birational}}"] & X_2 \ar[d, "{f_2}"'] \\
 Y  & & Y_5 \ar[ll, "{\text{finite}}"'] & Y_2. \ar[l, "{\text{birational}}"']
 \ar[lll, bend left=20, "{\gamma_2}"] 
\end{tikzcd}
\end{center}
By construction, the restriction of the finite morphism $Y_5 \to Y$ to the preimage $Y_5^\circ$ of $Y^\circ$ in $Y_5$ identifies with $Y_3^\circ \to Y^\circ$ and the restriction $f_6^\circ $ of $f_6$ to the preimage $X_6^\circ$ of 
$Y_5^\circ$ in $X_6$ identifies with $f_3^\circ$. 

Let $\sA_{3}$ be the fine moduli space of polarized abelian varieties with level three structure, and let $Y_2 \to Z_2 \to \sA_3$ be the Stein factorization of the morphism $Y_2 \to \sA_3$ defined by $f_2$. Note that $X_2 \to Y_2$ is the pull-back of an abelian scheme  over $Z_2$ with a level three structure via $Y_2 \to Z_2$. Let $C \subseteq  Y_2$ be a curve which is contracted in $Z_2$. Lemma \ref{lemma:trivial_family} then easily implies that $C$ must be contracted in $Y_5$ as well since any curve $B\subset X$ with $\dim f(B)=1$ has positive degree with respect to $K_X+D$ by construction. Hence, the morphism $Y_2 \to Y_5$ factors through the birational morphim $Y_2 \to Z_2$. Replacing $X$ by $X_6$ and $Y_2$ by $Z_2$ if necessary, we may therefore assume without loss of generality that we have a commutative diagram:
\begin{center}
\begin{tikzcd}[row sep=large, column sep=large]
X \ar[d, "{f}"']\ar[rr, dashed, "{\beta_2,\text{ birational}}"] && X_2 \ar[d, "{f_2}"'] \\
 Y && Y_2. \ar[ll, "{\gamma_2,\text{ birational}}"']
\end{tikzcd}
\end{center}
and that the morphism $Y_2 \to \sA_3$ defined by $f_2$ is finite.

\medskip

\noindent\textit{Step 6.} We now prove that $K_{Y_2}$ is $\mathbb{Q}$-ample. We will need the following observation. 
\begin{claim}\label{claim:exceptional_locus}
The rational map $\beta_2\colon X \map X_2$ is a morphism and $\textup{Exc}\,\beta_2 \subseteq D$.
\end{claim}

\begin{proof}
Recall from \cite[Lemma 5.9.3]{kollar_sh_inventiones} that $\sA_3$ does not contain any rational curve. The first claim then follows from \cite[Corollary 1.5]{hacon_mckernan}. 

Notice that there exists an effective Cartier divisor $E$ on $X$ such that $E\cdot C <0$ for any curve $C$ contracted by $\beta_2$ (see \cite[Paragraph 1.42]{debarre}). The second claim then follows from Lemma \ref{lemma:neagtivity_normal_bundle}.
\end{proof}

Observe that $K_X+D\sim_\mathbb{Q}\beta_2^*(K_{X_2}+D_2)$ where $D_2:=(\beta_2)_*D$. It follows that the pair $(X_2,D_2)$ has log canonical singularities. Moreover, $K_{X_2}+D_2$ is the pull-back of an ample divisor on $Y$. Since $f_2$ is a smooth morphism and $D_2=f_2^{-1}(f_2(D_2))$, any log canonical center of $(X_2,D_2)$ is the preimage of a subvariety of $Y_2$. Let $i \in I$ and set $A_i:=\beta_2(D_i)$. By Lemma \ref{lemma:step1} and \cite[Corollary 1.7]{druel_lo_bianco}, 
there exists a smooth morphism $D_i \to T_i$ with rationally connected fibers onto a finite \'etale quotient of an abelian variety. As a consequence,
the morphism $D_i \to A_i$ is smooth. Moreover, by
\cite[Lemma 5.9.3]{kollar_sh_inventiones}, we must have $\dim f_2(A_i)=0$ using the fact that 
$Y_2 \to \sA_3$ is finite. Since $A_i$ is a log canonical center of $(X_2,D_2)$ we conclude that $A_i$ is a fiber of $f_2$. 

\begin{claim}\label{claim:base_singular_locus}
The variety $Y_2$ is smooth away from the finitely many points $f_2(A_i)$. Moreover, $X$ is smooth.
\end{claim}

\begin{proof}
By Claim \ref{claim:exceptional_locus}, $X \setminus D \cong X_2 \setminus \beta_2(D)$. This implies that 
$X_2 \setminus \beta_2(D)$ has isolated singularities. This in turn implies that 
$X \setminus D$ and $\setminus \beta_2(D)$ are smooth since $\dim Y_1<\dim X_2$ by assumption.
Our claim follows easily.
\end{proof}

If $\dim Y_2=1$, then $\gamma_2$ is an isomorphism and $K_{Y_2}$ is ample by \cite[Lemma 5.9.3]{kollar_sh_inventiones} again. Suppose from now on that $\dim Y_2 \ge 2$. Then $D_2=0$ and $\textup{Exc}\,\beta_2 = D$. 
Let $s_2\colon Y_2 \to X_2$ be the unit section. Recall that $X_2$ is log canonical.
Since $f_2$ is smooth, $K_{Y_2}$ is $\mathbb{Q}$-Cartier and $Y_2$ is likewise log canonical. Moreover, $C \subset Y_2$ is a log canonical center of $Y_2$ if and only if $f_2^{-1}(C)$ is a log canonical center of $X_2$. It follows that the log canonical centers of $Y_2$ are the finitely many points $f_2(A_i)$ with $i \in I$.

Let $\wb{Y}_2 \subset X_2$ be the closed subscheme defined by the ideal $\beta_2^{-1}(\sI_{s_2(Y_2)})\cdot\sO_{X_2}$ and let $\eta_2\colon \wb{Y}_2 \to Y_2$ be the natural morphism.

\begin{claim}\label{claim:resolution}
The morphism $\eta_2$ is a resolution of $Y_2$. Moreover, the pair $(\wb{Y}_2,\wb{D}_2)$ is log smooth and $K_{\wb{Y}_2}+\wb{D}_2 \sim_\mathbb{Q} \eta_2^*K_{Y_2}$, where $\wb{D}_2:=D|_{\wb{Y}_2}$. In particular, $\wb{D}_2:=D|_{\wb{Y}_2}$ is reduced.
\end{claim}

\begin{proof}
Notice first that $\eta_2$ is an isomorphism away from the points $f_2(A_i)$. 
It follows from Claim \ref{claim:base_singular_locus} that $\wb{Y}_2$ is smooth away from $\wb{Y}_2\cap D$.
Set $g:=\dim X_2 - \dim Y_2=\dim X - \dim Y$.
Let $y_i$ be the unique point in $s_2(Y_2)\cap A_i$. There exit local functions $t_1,\ldots,t_g$ at $y_i$
such that $s_2(Y_2)$ is defined by the equations $t_1=\cdots=t_g=0$ and such that the restrictions of the $d_{y_i}t_j$ to $T_{y_i}A_i$
are linearly independent. Together with the fact that the morphism $D_i \to A_i$ is smooth,
this implies that the pair $(\wb{Y}_2,\eta_2^{-1}(y_i))$ is log smooth in a neighborhood of $\eta_2^{-1}(y_i)$. 

An easy computation using the adjunction formula then shows that $K_{\wb{Y}_2}+\wb{D}_2 \sim_\mathbb{Q} \eta_2^*K_{Y_2}$. It follows that $\wb{D}_2$ is reduced since $Y_2$ is log canonical, finishing the proof of the claim.
\end{proof}

\begin{claim}
The canonical divisor $K_{Y_2}$ is $\mathbb{Q}$-ample.
\end{claim}

\begin{proof}
Recall that the morphism $Y_2 \to \sA_3$ defined by $f_2$ is finite by construction.

By the cone theorem for log canonical spaces (see \cite[Theorem 1.4]{fujino_non_vanishing}), if $K_{Y_2}$ is not nef, then $Y_2$ contains a rational curve. But this contradicts \cite[Lemma 5.9.3]{kollar_sh_inventiones} and shows that $K_{Y_2}$ is nef.

By \cite[Lemma 5.9.3]{kollar_sh_inventiones} again, $\kappa(\wb{Y}_2)=\dim Y_2$. Thus 
$$\dim Y_2=\kappa(\wb{Y}_2)\le \kappa(K_{\wb{Y}_2}+\wb{D}_2)=\kappa(K_{Y_2})$$ 
using Claim \ref{claim:resolution}. It follows that $K_{Y_2}$ is abundant. 
Recall that the log canonical centers of $Y_2$ are the points $f_2(A_i)$ so that 
$K_{Y_2}$ is automatically log abundant. As a consequence of \cite[Theorem 1.6]{fujino_gongyo_14}, the canonical divisor $K_{Y_2}$ is semiample.

Let $g \colon Y_2 \to I$ be the birational morphism defined by the linear system $|mK_{Y_2}|$ for $m \ge 1$ sufficiently divisible. We argue by contradiction and assume that there exists a curve $C \subset Y_2$ such that $K_{Y_2}\cdot C=0$.

Recall that $\eta_2$ induces an isomorphism $\wb{Y}_2\setminus \wb{D}_2\cong Y_2 \setminus \eta_2(\wb{D}_2)$. It follows that the strict transform $\wb{C}$ of $C$ in $\wb{Y}_2$ is not contained in $\wb{D}_2$. As a consequence, either 
$K_{\wb{Y}_2}\cdot \wb{C}<0$ or $\wb{C}\cap \wb{D}_2=\emptyset$. Hence, either $K_{\wb{Y}_2}\cdot \wb{C}<0$ or $\textup{Exc}\, g \subset Y_2 \setminus \eta_2(\wb{D}_2)$.

Suppose first that $K_{\wb{Y}_2}\cdot \wb{C}<0$. A theorem of Miyaoka and Mori then implies that there is a rational curve through any point of $\wb{C}$. It follows that $Y_2$ contains rational curves, yielding a contradiction.

Suppose now that $\textup{Exc}\, g \subset Y_2 \setminus \eta_2(\wb{D}_2)=:Y_2^\circ$. There exists an effective Cartier divisor $E^\circ$ on $Y_2^\circ$ such that $E^\circ\cdot C <0$. If $\epsilon \in \mathbb{Q}_{>0}$ is small enough, then the pair $(Y_2^\circ,\epsilon E^\circ)$ is klt. The relative cone theorem for klt spaces applied to $Y_2^\circ \to g(Y_2^\circ)$ then shows that $Y_2$ contains a rational curve, yielding again a contradiction. This completes the proof of the claim.
\end{proof}

\medskip

\noindent\textit{Step 7.} The final step of the argument is similar to that of \cite[Theorem 6.1]{jahnke_radloff_13}.
Let now $C \subset Y_2$ be a general complete intersection curve of members of $|m(K_{Y_2})|$. By general choice of $C$, we have $C \subset Y_2 \setminus \eta_2(\wb{D}_2)$. Set $B:=s_2(C) \subset X_2$.
Set also $\sE^{1,0}:=(f_2)_*\Omega^1_{X_2/Y_2}$ so that $\Omega^1_{X_2/Y_2}\cong f_2^*\sE^{1,0}$. Consider the exact sequence
$$0 \to (f_2|_{X_2^\circ})^*\Omega^1_{Y_2^\circ}\to \Omega^1_{X_2^\circ}\to \Omega^1_{X_2^\circ/Y_2^\circ}\cong
(f_2|_{X_2^\circ})^*(\sE^{1,0}|_{Y_2^\circ}) \to 0.$$
Set $d:=\dim Y_2$. By Theorem \ref{thm:JR_singular}, the vector bundle $\Omega^1_{X_2}|_B$ is semistable. It follows that  
$$\frac{K_{Y_2}^d+\deg_C \sE^{1,0}}{n}=\frac{\deg_B K_{X_2}}{n} \le \frac{\deg_B f_2^*\sE^{1,0}}{g}=\frac{\deg_C \sE^{1,0}}{g}$$
and hence
$$\frac{K_{Y_2}^d}{d}\le \frac{\deg_C \sE^{1,0}}{g}.$$
On the other hand, by \cite[Theorem 1 and Remark 2]{viehweg_zuo_arakelov} applied to the pull-back of $f_2$ to $\wb{Y}_2$, we have 
$$2\frac{\deg_C \sE^{1,0}}{g} \le \frac{(K_{\wb{Y}_2}+\wb{D}_2)^d}{d}=\frac{K_{Y_2}^d}{d}.$$
This yields a contradiction since $K_{Y_2}^d>0$, finishing the proof of the theorem. 
\end{proof}


\newcommand{\etalchar}[1]{$^{#1}$}
\providecommand{\bysame}{\leavevmode\hbox to3em{\hrulefill}\thinspace}
\providecommand{\MR}{\relax\ifhmode\unskip\space\fi MR }
\providecommand{\MRhref}[2]{%
  \href{http://www.ams.org/mathscinet-getitem?mr=#1}{#2}
}
\providecommand{\href}[2]{#2}

\end{document}